\documentclass[11pt,reqno,oneside]{amsart}

\usepackage{graphicx}
\usepackage{amssymb}
\usepackage{epstopdf}

\usepackage[a4paper, total={6in, 9.1in}]{geometry}

\usepackage{amsmath,amsfonts,amsthm,mathrsfs,amssymb,cite}
\usepackage[usenames]{color}

\usepackage{subfigure} 
\usepackage{framed}

\usepackage{enumerate}
\usepackage{hyperref}
\usepackage{xcolor}
\hypersetup{
  colorlinks,
  linkcolor={blue!80!black},
  urlcolor={blue!80!black},
  citecolor={blue!80!black}
}
\usepackage[capitalize,noabbrev]{cleveref}

\usepackage{booktabs} 
\usepackage{array} 
\usepackage{paralist} 
\usepackage{verbatim} 
\usepackage{tabularx}
\usepackage{amsmath,amsfonts,amsthm,mathrsfs,amssymb,cite}
\usepackage[usenames]{color}
\usepackage{bm}

\newtheorem{thm}{Theorem}[section]

\newtheorem{lem}{Lemma}[section]

\theoremstyle{definition}

\theoremstyle{remark}

\newtheorem{rem}{Remark}[section]
\numberwithin{equation}{section}

\allowdisplaybreaks

\title[Time-domain direct sampling method in electromagnetism]{Time-domain direct sampling method for inverse electromagnetic scattering  with a single incident source
}

\author{Chen Geng}
\address{School of Mathematics, Harbin Institute of Technology, Harbin, P. R. China.}
\email{gengchen@stu.hit.edu.cn}

\author{Minghui Song}
\address{School of Mathematics, Harbin Institute of Technology, Harbin, P. R. China.}
\email{songmh@hit.edu.cn}

\author{Xianchao Wang}
\address{School of Mathematics, Harbin Institute of Technology, Harbin, P. R. China.}
\email{xcwang90@gmail.com, xcwang@hit.edu.cn}

\author{Yuliang Wang}
\address{Research Center for Mathematics, Beijing Normal University, Guangdong Provincial Key Laboratory IRADS, BNU-HKBU United International College,  Zhuhai 519087, China. }
\email{yuliangwang@bnu.edu.cn}

\begin{document}
\maketitle

\begin{abstract}
	
	In this paper, we consider an inverse electromagnetic medium scattering problem of reconstructing unknown objects from time-dependent boundary measurements.
	A novel time-domain direct sampling method is developed for determining the locations of unknown scatterers by using only a single incident source. Notably, our method imposes no restrictions on the  the waveform of the incident wave.
	Based on the Fourier-Laplace transform,  we first establish the connection between the frequency-domain  and the time-domain direct sampling method. Furthermore, we elucidate the mathematical mechanism of the imaging functional through the properties of modified Bessel functions. Theoretical justifications and stability analyses are provided to demonstrate the effectiveness of the proposed method. Finally, several numerical experiments are presented to illustrate the feasibility of our approach.

	\medskip

	\noindent{\bf Keywords}: inverse electromagnetic scattering problem, time domain, direct sampling method, Fourier-Laplace transform, modified Bessel functions

\end{abstract}

\section{Introduction}

\subsection{Background and motivation}
This paper  is concerned with  an inverse electromagnetic medium scattering problem, which aims to determine unknown scatterers from boundary measurements of electromagnetic waves generated by a given incident wave. Typically, an incident wave is directed towards the objects of interest, and an array of receivers is positioned on the boundary, away from these objects. The measured electromagnetic scattering waves from these receivers are then utilized to identify the locations and shapes of the objects. This type of inverse problem has attracted significant attention in the literature due to its practical applications in various fields of science and technology, such as radar imaging, biomedical diagnosis \cite{Arridge1999, Deng2019}, and non-destructive testing \cite{Ammari2013}.

The mathematical theory underlying this inverse scattering problem has been extensively investigated and established in the frequency domain \cite{Ammari2018, Colton2019}.  The primary challenge in solving this inverse problem lies in the inherent non-linearity and high ill-posedness of reconstructing the unknown medium from boundary measurements.
To address this challenge, a variety of frequency-domain numerical reconstruction methods have been developed in recent decades.
Broadly speaking,  numerical approaches can be categorized into two types: quantitative and qualitative approaches.
Quantitative methods primarily employ linearization and optimization strategies to directly determine the parameters of unknown scatterers. Typical methodologies in this category include Newton's iterative method  and the decomposition method \cite{Hagemann2019, Kress2007, Zhang2017} .
Although such quantitative methods could yield good reconstructions, they usually require a reliable initial guess and also involve expensive computation, especially in three-dimensional electromagnetic scattering models.
To overcome these issues, numerous qualitative methods have been proposed in the literature that do not require an initial guess of the scatterers. These approaches rely on establishing criterion to distinguish the interior and exterior of unknown scatterers, thereby enabling a qualitative reconstruction of their support\cite{Ammari2013a}.  We refer the readers to the linear sampling method\cite{Cakoni2011}, the factorization method \cite{Kirsch2007}, the probe method\cite{Nakamura2015}, the direct sampling method \cite{Harris2022, Ito2013, Li2013,  park2018, Sun2019}, and the monotonicity method \cite{Griesmaier2018}.  It is noted that all the aforementioned methods focus on recovering unknown objects using single-frequency, time-harmonic waves. Furthermore, to improve imaging resolution, a wide range of multi-frequency recovery techniques  are employed due to their superior imaging performance in practice\cite{Liu2019, Zhang2015}. Typically, lower frequency reconstructions are used as initial guesses or {\it  a prior} information, while high-frequency data are utilized to enhance imaging resolution.
In particular,  multi-frequency reconstruction algorithms could achieve global convergence for the nonlinear inverse medium problem. For more details, interested readers could refer to the recursive linearization method\cite{Bao2015, Bao22015}.

However, in many practical applications, such as ultrasound imaging and earthquake seismology, acquiring time-dependent signals is often more feasible and straightforward.
While one can convert time-domain signals to the frequency domain by the Fourier transform and subsequently reconstruct scatterers via the aforementioned frequency-domain methods, this technique is effective only for steady-state or periodic signals. Conversely, time-domain methods directly utilize temporal information without the need for transformations, making them advantageous for real-time applications. Moreover, time-domain methods can be applied to more complex scenarios where scatterer parameters change over time. For further information on time-domain methods, readers could refer to the synthetic aperture radar method, the time reversal method, and the total focusing method \cite{Fink1997, Ammari2013, Holmes2005, Sini2022}. The imaging mechanism of these time-domain methods is based on the travel time of recorded scattered waves. Nevertheless, in comparison of frequency-domain methods,  existing time-domain methods usually ignore the echo amplitude of the scattered field.  Hence, there is considerable interest in developing a time-domain approach that incorporates both travel time and echo amplitude information for solving the inverse electromagnetic medium scattering problem.

\subsection{Mathematical modeling and imaging functional}
In the current article, we consider a time-dependent inverse electromagnetic medium scattering problem.
To begin with, we introduce the mathematical formulations for the interaction of electromagnetic wave with an inhomogeneous medium.
Let $\epsilon_0$ and $\mu_0$ respectively signify  the electric permittivity and the magnetic permeability in vacuum.
Suppose that the scatterers consist of a dielectric material, with electric permittivity $\epsilon(x)\in L^{\infty}(\mathbb{R}^3)$, magnetic permeability  $\mu(x)\equiv \mu_0$ and  zero electric conductivity. We also assume that  $\epsilon(x)-\epsilon_0$ is compactly supported in $\Omega\in \mathbb{R}^3$ with $\mathbb{R}^3\backslash \Omega$ being connected.
Let the incident wave $\mathcal{E}^i(x,t)$ be  a causal signal (that is, $\mathcal{E}^i(x,t)\equiv 0$ for $t\leq0$) emitted from a fixed point $x\in \mathbb{R}^3\backslash \Omega$, then  the propagation of electromagnetic scattered waves $\mathcal{E}^s$ and $\mathcal{H}^s$ are governed by the following  system:
\begin{equation}\label{eq1}
\begin{aligned}
&\nabla\times\mathcal{E}^s+\mu_0\frac{\partial \mathcal{H}^s}{\partial t}=0, \quad
\nabla\times \mathcal{H}^s-\epsilon(x)\frac{\partial \mathcal{E}^s}{\partial t}=(\epsilon(x)-\epsilon_0)\frac{\partial \mathcal{E}^i}{\partial t}, \quad (x,t)\in \mathbb{R}^3\times\mathbb{R},\\
&\mathcal{E}^s(\cdot, 0)=0, \quad \mathcal{H}^s(\cdot, 0)=0, \quad x\in  \mathbb{R}^3.
\end{aligned}
\end{equation}
The well-posedness of the forward time-domain electromagnetic scattering problem \eqref{eq1} is thoroughly discussed in \cite{Lahivaara2022}.
Let $\Gamma\subset\mathbb{R}^3\backslash \overline{\Omega}$ be a closed or open measurement surface, where the receivers are located away from the targets. The inverse problem is to determine the unknown/inaccessible scatterers $\Omega$ from the boundary measurements:
\begin{align*}
\Lambda:=\{\mathcal{E}^s(x,t):   \  (x,t)\in\Gamma\times\mathbb{R}^+\}.
\end{align*}

To identify the unknown scatterer, the time-domain Newton iterative method is proven to be an efficient approach \cite{Zhao2022}. However, such iterative methods typically incur significant computational costs. To overcome this difficulty, some existing time-domain methods in the literature focus on fast imaging techniques, which is similar to the frequency-domain qualitative approaches.
Although time-domain methods have been successful and widely used in practical applications, their theoretical analysis remains significantly less complete compared to frequency-domain methods. To this end,  frequency-domain linear sampling method and factorization method were extended  to the time domain case\cite{Cakoni2019, Haddar2010, Haddar2020,  guo2013}.  Moreover,  rigorous mathematical justifications were established for the time-dependent inverse acoustic scattering problem via the Fourier-Laplace transform.  It is worthy to mention that both time-domain linear sampling method and factorization method suffer from theoretical difficulties related to transmission eigenvalue problems. While the solvability of the time-domain linear sampling method has been rigorously proven in the context of acoustic waves \cite{cakoni2021}, its corresponding solvability for electromagnetic waves is still an open issue \cite{Lahivaara2022}.

In order to avoid the transmission eigenvalue problems,  inspired by  \cite{guo2023}, we  propose a novel time-domain direct sampling method for time-domain electromagnetic scattering problem  and  reveal its mathematical mechanism. To this end, we first introduce the imaging functional for  the inverse electromagnetic medium scattering problem:
\begin{equation}\label{ind}
\mathcal{I}(z)=\int_{-\infty}^{+\infty}\bigg|\int_{\Gamma}\mathcal{E}^s(x, t+c_0^{-1}|x-z|)\frac{\mathrm{e}^{-\sigma (t+ c_0^{-1}|x-z|)}}{4\pi|x-z|}\mathrm{d}s(x)\bigg|^2 \mathrm{d}t,\quad z\in \tilde{\Omega},
\end{equation}
where $\Gamma$ signifies the measurement surface, $c_0^{-1}=\sqrt{\epsilon_0 \mu_0}$ is the reciprocal of the speed of light in vacuum, $\sigma>0$ is a constant associated with the Fourier-Laplace transform, and $\tilde{\Omega} \supset \Omega$ is the sampling domain.

According to formula \eqref{ind}, one can find that the imaging functional involves a time-space integral of retarded scattered waves with a retarded test function. Moreover, the term $c_0^{-1}|x-z|$ represents the delay travel time information of the scattered waves. Thus, the proposed imaging functional not only  utilizes the echo amplitude of the scattered field but also incorporates the travel time information of the recorded  data. We would like to emphasize that the improper time integral of the imaging function can be approximated with one in a finite interval by choosing a suitable $\sigma$.  On the other hand,  to clarify the mathematical mechanism of the imaging functional \eqref{ind}, we first apply the Fourier-Laplace transform to convert the time-domain imaging function into the frequency domain.  Then we analyze the asymptotic behavior of the transformed imaging functional by using  the properties of the modified Bessel functions. Finally, we show that the support of the scatterers can be determined based on the behavior of the imaging functional.

The promising features of our proposed time-domain direct sampling method can be summarized as follows.
The imaging functional involves a time-space convolution between the measured time-dependent scattered data and a specific test function. Consequently, this functional can be computed with high efficiency, potentially enabling real-time imaging for practical applications.
Secondly, compared to the total focusing method \cite{Holmes2005}, our imaging functional not only utilizes the spatial integral but also incorporates the temporal integral. Thus, the total focusing method can be regarded as a special case of our proposed approach.
Thirdly,  most existing time-domain methods require the incident wave to be a Gaussian pulse wave. However, our proposed method has no restrictions on the location and waveform of the incident wave. We can ensure the integral \eqref{ind} to converge by adjusting the parameter $\sigma$. Finally, in comparison of the frequency-domain direct sampling method, the proposed time-domain direct sampling method does not require knowledge of the polarization direction of the incident electromagnetic wave. Therefore, our imaging scheme is more flexible and applicable across a wider range of scenarios.

The outline of this paper is as follows.  In Section 2, we introduce the Fourier-Laplace transform and the corresponding Maxwell system in the frequency domain. In Section 3, we establish a one-to-one relationship between time-domain and frequency-domain direct sampling based on the Fourier-Laplace transform. Furthermore, we analyze the asymptotic behavior and stability of the proposed imaging functional. Several numerical experiments are presented in Section 4 to demonstrate the validity of the proposed  time-domain direct sampling method.

\section{Laplace transform and well-posedness of the forward problem}

In this section, we introduce some notations for the space-time Sobolev spaces and the norms involved in the subsequent discussion. To this end,  we first provide a brief introduction of the Laplace transform and the associated Sobolev spaces. Assume that $X$ is a Hilbert space. The set of $X$-valued test functions on the real line with compact support in $(-\infty,+\infty)$ is defined by $\mathcal{D}(\mathbb{R}, X)$, and the associated $X$-valued distributions are defined by $\mathcal{D}^{\prime}(\mathbb{R}, X)$. The Schwartz space of $X$-valued $C^\infty$ functions on the real line is defined by $\mathcal{S}(\mathbb{R}, X)$, and the associated tempered distributions are defined by $\mathcal{S}^{\prime}(\mathbb{R}, X)$. Following \cite{Sayas},  we define
\begin{align*}
\mathcal{L}_\sigma^{\prime}(\mathbb{R}, X):=\{f\in\mathcal{D}^{\prime}(\mathbb{R}, X):\mathrm{e}^{-\sigma t}f\in\mathcal{S}^{\prime}(\mathbb{R}, X)\},\ \ \sigma\in\mathbb{R}.
\end{align*}
Let $\mathbb{C}_{\sigma_0}=\{\omega\in\mathbb{C}:\mathrm{Im}(\omega)\geq\sigma_0>0\}$ be a half-plane with $\omega=\xi+\mathrm{i}\sigma$, then the Fourier-Laplace transform  of a function $f\in\mathcal{L}_\sigma^{\prime}(\mathbb{R}, X)$ is defined by
\begin{align*}
\mathscr{L}[f](\omega):=\int_{-\infty}^{+\infty}\mathrm{e}^{\mathrm{i}\omega t}f(t) dt, \ \ \omega\in\mathbb{C}_{\sigma_0}.
\end{align*}
For $m\in\mathbb{R}$,  we also introduce the Hilbert space
\begin{align*}
H_\sigma^m(\mathbb{R},X):=\left\{f\in\mathcal{L}_\sigma^{\prime}(\mathbb{R},X):\int_{-\infty+i\sigma}^{+\infty+i\sigma}|\omega|^{2m}\left\|\mathscr{L}[f](\omega)\right\|_X^2\mathrm{d}\omega<+\infty\right\},
\end{align*}
endowed with the norm
\begin{equation*}
  \|f\|_{H_\sigma^m(\mathbb{R},X)}=\left(\int_{-\infty+i\sigma}^{+\infty+i\sigma}|\omega|^{2m}\left\|\mathscr{L}[f](\omega)\right\|_X^2\mathrm{d}\omega\right)^{1/2}.
\end{equation*}
By Plancherel theorem \cite{cohen2007},  the Fourier-Laplace transform satisfies the following Parseval's identity
\begin{equation}\label{eq:Parseval}
\int_{-\infty}^{+\infty} |\mathrm{e}^{-\sigma t}f(t)|^2\, \mathrm{d}t = \frac{1}{2\pi} \int_{-\infty+\mathrm{i}\sigma}^{+\infty+\mathrm{i}\sigma} \left|\mathscr{L}[f](\omega)\right|^2 \, \mathrm{d}\omega.
\end{equation}



To facilitate theoretical analysis, we will convert the time domain problem \eqref{eq1} into the frequency domain.
By applying the Fourier-Laplace transform, the system $\eqref{eq1}$ can be represented by
\begin{equation*}\label{maxqq}
\begin{aligned}
& \nabla\times {\hat{E}^s}-\mathrm{i}\omega\mu_0 {\hat{H}^s}=0, \quad  \nabla\times {\hat{H}^s}+ \mathrm{i}\omega\epsilon(x){\hat{E}^s}=-\mathrm{i}\omega\big(\epsilon(x)-\epsilon_0\big)\hat{E}^i,
\ \ (x,t)\in\mathbb{R}^3\times\mathbb{R},
\end{aligned}
\end{equation*}
where $\hat{E}^s=\mathscr{L}[\mathcal{E}^s]$, $\hat{H}^s=\mathscr{L}[\mathcal{H}^s]$ and $\hat{E}^i=\mathscr{L}[\mathcal{E}^i]$. By eleminating the magnetic field $\hat{H}^s$, we obtain the following equation
\begin{align}\label{maxq}
\nabla\times(\nabla\times {\hat{E}^s})-\omega^2\mu_0\epsilon(x){\hat{E}^s}=\omega^2\mu_0\big(\epsilon(x)-\epsilon_0\big)\hat{E}^i,\ \ x\in\mathbb{R}^3.
\end{align}
Recalling  the Lippmann-Schwinger integral equation for Maxwell system \cite{Ito2013},  the solution to the last equation is given by
\begin{align}\label{eqs}
\hat{E}^s(x,\omega)= \int_{\mathbb{R}^3}\omega^2\mu_0\big(\epsilon(x)-\epsilon_0\big)   {\hat{\Phi}}_{\omega}(x,y){\hat{E}}(y,\omega)\,\mathrm{d}y.
\end{align}
Here $\hat E:=\hat E^i+\hat E^s$ denotes the total field and
\begin{equation}\label{fund}
\hat \Phi_{\omega}(x,y)= \left(\mathbb{I}+ \frac{1}{\omega^2 \mu_0\epsilon_0}\nabla \nabla^{\top} \right)  \frac{\mathrm{e}^{\mathrm{i} \omega \sqrt{\mu_0\epsilon_0} |x-y|}}{4\pi|x-y|},
\end{equation}
denotes the fundamental solution to
\begin{equation*}\label{prop1}
\nabla\times(\nabla\times{\hat{\Phi}}_{\omega}(x,y))-\omega^2 \mu_0\epsilon_0 {\hat{\Phi}}_{\omega}(x,y)= -\delta(x-y)\mathbb{I},
\end{equation*}
where $\mathbb{I}$ denotes the $3\times 3$ identity matrix and $\nabla \nabla^{\top}$ denotes the Hessian matrix.

In what follows, we shall discuss the well-posedness of the forward problem through the Fourier-Laplace transform. Before proceeding with our discussion, we introduce some notations and relevant spaces. We define the Sobolev space $H(\mathrm{curl}, \cdot)$ as
\begin{align*}
H(\mathrm{curl}, \Omega):=\{u\in (L^2(\Omega))^3:\, \nabla\times u\in(L^2(\Omega))^3\},
\end{align*}
with the associated norm
\begin{align*}
\|u\|_{H(\mathrm{curl}, \Omega)}=\left(\|u\|^2_{(L^2(\Omega))^3}+\|\nabla\times u\|^2_{(L^2(\Omega))^3}\right)^{1/2},
\end{align*}
and the similar definition for  $H(\text{curl},\mathbb{R}^3)$.


We introduce the Paley-Wiener theorem, which ensures the validity of analyzing our method in the frequency domain. For further details, see  \cite{lubich1994}.
\begin{lem}\label{pal}
Let $\omega\in\mathbb{C}_{\sigma_0}\mapsto f(\omega)\in\mathcal{B}(X,Y)$ be a holomorphic function with values of linear bounded operators between two Banach space $X$ and $Y$, satisfying
\begin{align*}
\|f(\omega)\|_{\mathcal{B}(X,Y)}\leq C|\omega|^r,\quad \omega\in\mathbb{C}_{\sigma_0},\,r\in\mathbb{R}.
\end{align*}
Let $\displaystyle{ F(t):=\frac{1}{2\pi}\int_{-\infty+\mathrm{i}\sigma}^{+\infty+\mathrm{i}\sigma}\mathrm{ e}^{-i\omega t}f(\omega)\, \mathrm{ d}\omega}$ and $\displaystyle{\mathcal{F}g:=\int_{-\infty}^{+\infty}F(\tau)g(t-\tau) \, \mathrm{ d}\tau}$ be the associated convolution operator. Then, for any $m\in\mathbb{R}$, $\mathcal{F}$ extends to a bounded operator from $H_\sigma^{m+r}(\mathbb{R},X)$ to $H_\sigma^m(\mathbb{R},Y)$.
\end{lem}


With the above results, the solvability of the forward problem \eqref{eq1} is establish in the following theorem.

\begin{thm}
Let $m\in\mathbb{R}$ and $\sigma>0$. If $\mathcal{E}^i(x,t)\in H_\sigma^m(\mathbb{R},(L^2(\Omega))^3)$, then the system (\ref{eq1}) has a unique solution $\mathcal{E}^s(x,t)\in H_\sigma^{m-1}(\mathbb{R},H(\mathrm{curl},\mathbb{R}^3))$. Moreover, the solution $\mathcal{E}^s$ is bounded in terms of the incident field $\mathcal{E}^i$
\begin{align*}
\|\mathcal{E}^s\|_{H_\sigma^{m-1}(\mathbb{R},H(\mathrm{curl},\mathbb{R}^3))}\leq C\|\mathcal{E}^i\|_{H_\sigma^{m}(\mathbb{R},(L^2(\Omega))^3)},
\end{align*}
where $C>0$ is a constant.
\end{thm}
\begin{proof}
Multiplying both sides of equation \eqref{maxq} by the complex conjugate of a smooth test function $V\in (C_0^\infty(\mathbb{R}^3))^3$, one can obtain the following variational function
\begin{align}\label{variation}
\int_{\mathbb{R}^3}(\nabla\times(\nabla\times \hat{E}^s)-\omega^2\mu_0\epsilon(x)\hat{E}^s)\cdot \overline{V} \, \mathrm{d}x
=\omega^2\mu_0\int_{\Omega}(\epsilon(x)-\epsilon_0)\hat{E}^i\cdot \overline{V}\, \mathrm{d}x.
\end{align}	
We define a bilinear form as follows:
\begin{align}\label{bilinear}
A(U,V):=\int_{\mathbb{R}^3}(\nabla\times U)\cdot(\nabla\times \overline{V})-\omega^2\mu_0\epsilon(x)U\cdot \overline{V}\mathrm{ d}x.
\end{align}
Using Green's integral formula, the variational function \eqref{variation} can be expressed as:
\begin{equation}\label{variation2}
A(\hat{E}^s,V)=\omega^2\mu_0\int_{\Omega}(\epsilon(x)-\epsilon_0)\hat{E}^i\cdot \overline{V}\mathrm{ d}x,
\end{equation}
  Due to $\epsilon(x)\in L^{\infty}(\mathbb{R}^3)$,  we assume that the electric permittivity in the scatterers has a positive  bound, such that $0<\epsilon_{\min}\leq\epsilon(x)\leq \epsilon_{\max}$.
By setting $U=V$ in \eqref{bilinear}, we have
\begin{align*}
\Re\big(\mathrm{i}\overline{\omega}A(U,U)\big)&=\Re\bigg(\int_{\mathbb{R}^3}\Big(\mathrm{i}\overline{\omega}(\nabla\times U)\cdot(\nabla\times \overline{U})-\mathrm{i}|\omega|^2\omega\mu_0\epsilon(x)U\cdot \overline{U}\Big)\mathrm{d}x\bigg)\\
&=\text{Im}(\omega)\int_{\mathbb{R}^3}\Big((\nabla\times U)\cdot(\nabla\times \overline{U})+|\omega|^2\mu_0\epsilon(x)U\cdot \overline{U}\Big)\mathrm{ d}x\\
&\geq \sigma\cdot\text{min}\{1,|\omega|^2\mu_0\epsilon_{\text{min}}\}\|U\|^2_{H(\text{curl},\mathbb{R}^3)}.
\end{align*}
On the other hand, it is easy to verify that
\begin{equation*}
  \Re\big(\mathrm{i}\overline{\omega}A(U,V)\big)\leq \sigma \cdot \left(1+|\omega|^2 \mu_0 \epsilon_{\text{max}}\right)\|U\|^2_{H(\text{curl},\mathbb{R}^3)} \, \|V\|^2_{H(\text{curl},\mathbb{R}^3)}.
\end{equation*}
Therefore, the Lax-Milgram lemma ensures the existence of a unique solution $\hat{E}^s$ for the variational formulation of \eqref{maxq}.

Meanwhile, by \eqref{variation2}, we can derive that
\begin{align*}
\Re\big(\mathrm{i}\overline{\omega}A(\hat{E}^s,\hat{E}^s)\big)&=\Re\bigg(\mathrm{i}|\omega|^2\omega\mu_0\int_{\Omega}(\epsilon(x)-\epsilon_0)\hat{E}^i\cdot \overline{\hat{E}^s}\mathrm{ d}x\bigg)\\
&\leq|\omega|^3\mu_0 \|\epsilon(x)-\epsilon_0\|_{L^\infty(\Omega)}\bigg|\int_{\Omega}\hat{E}^i\cdot \overline{\hat{E}^s}\mathrm{ d}x\bigg|.
\end{align*}
By applying the Cauchy-Schwarz inequality, one can find that
\begin{align*}
0<\sigma\cdot\text{min}\{1,|\omega|^2\mu_0\epsilon_{\text{min}}\}\|\hat{E}^s\|^2_{H(\text{curl},\mathbb{R}^3)}&\leq C_1|\omega|^3\|\hat{E}^i\|_{(L(\Omega))^3}\|\hat{E}^s\|_{(L(\Omega))^3}\\
&\leq C_1|\omega|^3\|\hat{E}^i\|_{(L(\Omega))^3}\|\hat{E}^s\|_{H(\text{curl},\mathbb{R}^3)},
\end{align*}
where $C_1=\mu_0 \|\epsilon-\epsilon_0\|_{L^\infty(\Omega)}$. Therefore, we have
\begin{align*}
\|\hat{E}^s\|_{H(\text{curl},\mathbb{R}^3)}\leq\text{max}\left\{\frac{C_1}{\sigma\mu_0\epsilon_{\text{min}}}|\omega|,\,\frac{C_1}{\sigma}|\omega|^3\right\}\|\hat{E}^i\|_{(L(\Omega))^3}.
\end{align*}
We define the solution operator $\hat{S}(\omega)$ of the problem \eqref{maxq} that maps the incident field to scattered field
	\begin{equation*}
	\hat{S}(\omega): \hat{E}^i\in (L^2(\Omega))^3 \mapsto \hat{E}^s\in H(\mathrm{curl},\mathbb{R}^3),
	\end{equation*}
	and the solution operator $\mathcal{S}: H_\sigma^m(\mathbb{R},(L^2(\Omega))^3)\rightarrow H_\sigma^{m-1}(\mathbb{R},H(\mathrm{curl},\mathbb{R}^3))$ that can be interprated as
	\begin{equation*}
	\begin{aligned}
	\mathcal{S}(\mathcal{E}^i)&:=\mathscr{L}^{-1}[\hat{E}^s(\cdot, \omega)]
	=\mathscr{L}^{-1}[\hat{S}(\omega)\hat{E}^i(\cdot,\omega)]\\
	&=\int_{-\infty}^{+\infty}S(\tau)\mathcal{E}^s(\cdot, t-\tau)\mathrm{ d}\tau,
	\end{aligned}
	\end{equation*}
	where $\displaystyle{S(t):=\frac{1}{2\pi}\int_{-\infty+\mathrm{i}\sigma}^{+\infty+\mathrm{i}\sigma}\mathrm{ e}^{-i\omega t}\hat{S}(\omega)\mathrm{ d}\omega}$.
Hence, by lemma \ref{pal}, it follows that
\begin{align*}
\|\mathcal{E}^s\|_{H_\sigma^{m-1}(\mathbb{R},H(\mathrm{curl},\mathbb{R}^3))}\leq C\|\mathcal{E}^i\|_{H_\sigma^{m}(\mathbb{R},(L^2(\Omega))^3)}.
\end{align*}
\end{proof}

\section{Analysis of  the time-domain direct sampling method}\label{section3}

In this section, we will show the imaging mechanism of the proposed time-domain direct sampling method. The key ingredient is to transform the time-domain imaging function \eqref{ind}  into the frequency domain and then analyze the properties of the corresponding frequency-domain imaging function.  To this end,  we first utilize the Fourier-Laplace  transform  to convert the time-domain imaging function \eqref{ind} into the frequency domain.


\begin{thm}\label{thm:transform}
Let the time-domain indicator function  be defined  in  \eqref{ind}.
 Then, the equivalent frequency-domain  indicator function via the Fourier-Laplace transform is given by
\begin{equation}\label{find}
\mathcal{I}(z)=\frac{1}{2\pi}\int_{-\infty+ \rm{i}\sigma}^{+\infty+\rm{i}\sigma}\bigg|\int_{\Gamma} \hat{E}^s(x,\omega)\frac{\mathrm{e}^{-{\rm i}\Re(\omega) c_0^{-1}|x-z|}}{4\pi|x-z|}\,\mathrm{d}s(x)\bigg|^2\mathrm{d}\omega, \quad z\in D,
\end{equation}
where $\hat{E}^s=\mathscr{L}[\mathcal{E}^s]$.
\end{thm}
\begin{proof}
Using the Parseval's identity \eqref{eq:Parseval} and the property of the Fourier-Laplace transform
\begin{equation*}
  \mathscr{L}[f(t + a)](\omega) = \mathrm{e}^{-\mathrm{i}\omega a}\hat{f}(\omega), \quad a\in \mathbb{R},
\end{equation*}
it follows that
\begin{equation*}
\begin{aligned}
\mathcal{I}(z)&=\int_{-\infty}^{+\infty}\bigg|\int_{\Gamma}\mathcal{E}^s(x, t+  c_0^{-1}|x-z|)\frac{\mathrm{e}^{-\sigma (t+ c_0^{-1}|x-z|)}}{4\pi|x-z|}\mathrm{d}s(x)\bigg|^2\mathrm{d}t\\
&=\frac{1}{2\pi}\int_{-\infty+\mathrm{i}\sigma}^{+\infty+\mathrm{i}\sigma} \bigg| \mathscr{L}\left[\int_{\Gamma}\mathcal{E}^s(x, t+ c_0^{-1}|x-z|)\frac{\mathrm{e}^{-\sigma  c_0^{-1} |x-z|}}{4\pi|x-z|}\mathrm{d}s(x)\right](\omega)   \bigg|^2 \mathrm{d}\omega\\
&=\frac{1}{2\pi}\int_{-\infty+\mathrm{i}\sigma}^{+\infty+\mathrm{i}\sigma} \bigg| \int_{\Gamma} \mathrm{e}^{-\mathrm{i}\omega c_0^{-1}|x-z|}\hat{E}^s(x,\omega) \frac{\mathrm{e}^{-\sigma c_0^{-1}|x-z|}}{4\pi|x-z|}\mathrm{d}s(x)    \bigg|^2 \mathrm{d}\omega\\
&=\frac{1}{2\pi}\int_{-\infty+\mathrm{i}\sigma}^{+\infty+\mathrm{i}\sigma}\bigg|\int_{\Gamma} \hat{E}^s(x,\omega)\frac{\mathrm{e}^{-\text{i}\Re(\omega) c_0^{-1}|x-z|}}{4\pi|x-z|}\,\mathrm{d}s(x)\bigg|^2\mathrm{d}\omega.
\end{aligned}
\end{equation*}

This completes the proof.
\end{proof}

\begin{rem}
According to Theorem \ref{thm:transform}, one can find that the proposed time-domain imaging functional can be viewed as a  multi-frequency superposition of the frequency-domain imaging functional. To ensure the convergence of the equivalent frequency-domain imaging functional \eqref{find}, we need to constrain the frequency of the test function to be real numbers, that is, $\Re(\omega)$.  Therefore, the proposed time-domain indicator function \eqref{ind} is ingeniously designed.

\end{rem}

From Theorem \ref{thm:transform}, it is evident that the properties of the time-domain imaging functional
\eqref{ind} are identical to those of the frequency-domain imaging functional
\eqref{find}. Therefore, we will analyze the behavior of the frequency-domain imaging functional  to uncover the imaging mechanism of the time-domain direct sampling method. Before our discussion,  we give three crucial lemmas.

%
%
%

\begin{lem}\label{lem1}
Let $z\in\mathbb{R}^3$ and $\sigma>0$,  then it holds that
\begin{equation*}
\int_{\mathbb{S}^2}\left(\mathbb{I}-\hat{x}\hat{x}^{\top}\right) \mathrm{ e}^{\sigma c_0^{-1}  \hat{x}\cdot z} \mathrm{d}s(\hat{x})= \frac{8\pi}{3} i_0(\sigma c_0^{-1} |z|)\mathbb{I}+ \frac{4\pi}{3} (\mathbb{I}-3\hat{z}\hat{z}^{\top})i_2(\sigma c_0^{-1}|z|) ,
\end{equation*}
where $\mathbb{I}$ is  the $3\times 3$ identity matrix, $\mathbb{S}^2$ is the unit sphere in $\mathbb{R}^3$, and $i_n(\cdot)$ denotes the modified spherical Bessel functions of the first kind with order $n$.
\end{lem}
\begin{proof}
Let  $Y_n^m(\hat{x})$ be the spherical harmonics functions and $\hat{x}=(x_1,x_2,x_3)\in\mathbb{S}^2$. By the explicit expressions
\begin{equation*}
\begin{aligned}
&Y_2^0(\hat{x})=\frac{1}{4}\sqrt{\frac{5}{\pi}}(-x_1^2-x_2^2+2x_3^2),\quad
Y_2^2(\hat{x})=\frac{1}{4}\sqrt{\frac{15}{\pi}}(x_1^2-x_2^2),\\
&Y_2^{-2}(\hat{x})=\frac{1}{2}\sqrt{\frac{15}{\pi}},\quad
Y_2^{1}(\hat{x})=\frac{1}{2}\sqrt{\frac{15}{\pi}}x_1x_3,\quad
Y_2^{-1}(\hat{x})=\frac{1}{2}\sqrt{\frac{15}{\pi}}x_2x_3,
\end{aligned}
\end{equation*}
one can deduce that
\begin{equation}\label{eq:lem1}
\begin{aligned}
&x_1^2=\frac{1}{3}-\frac{2}{3}\sqrt{\frac{\pi}{5}}Y_2^0(\hat{x})+2\sqrt{\frac{\pi}{15}}Y_2^2(\hat{x}),\quad
x_2^2=\frac{1}{3}-\frac{2}{3}\sqrt{\frac{\pi}{5}}Y_2^0(\hat{x})-2\sqrt{\frac{\pi}{15}}Y_2^2(\hat{x}),\\
&x_1x_2=2\sqrt{\frac{\pi}{15}}Y_2^{-2}(\hat{x}),\quad
x_1x_3=2\sqrt{\frac{\pi}{15}}Y_2^{1}(\hat{x}),\quad
x_2x_3=2\sqrt{\frac{\pi}{15}}Y_2^{-1}(\hat{x}).
\end{aligned}
\end{equation}
Furthermore, due to Jacobi-Anger expansion given in \cite[p.37]{Colton2019} and the fact that $i_n(x)=i^{-n}j_n(\mathrm{ i}x)$, we have the following expansion
\begin{equation*}\label{eq:generating}
\mathrm{e}^{\sigma c_0^{-1} \hat{x}\cdot z}=4\pi \sum_{n=0}^{\infty}\sum_{m=-n}^{n}i_n(\sigma c_0^{-1}|z|)\overline{Y_n^m(\hat{x})}Y_n^m(\hat{z}),
\end{equation*}
where $j_n$  and $i_n$ denote the spherical Bessel function and the modified spherical Bessel function of the first kind of order $n$, respectively. Together with \eqref{eq:lem1},  it yields
\begin{equation}\label{eq:part}
\begin{aligned}
&\int_{\mathbb{S}^2}\mathrm{e}^{\sigma c_0^{-1}\hat{x}\cdot z}\mathrm{d}s(\hat{x})=4\pi\,  i_0(\sigma c_0^{-1}|z|),\\
&\int_{\mathbb{S}^2}x_p^2\cdot \mathrm{e}^{\sigma c_0^{-1}\hat{x}\cdot z}\mathrm{d}s(\hat{x})
=\frac{4\pi}{3}i_0(\sigma c_0^{-1}|z|)+\frac{4\pi}{3}(3 {\hat z}_p^2-1)i_2(\sigma c_0^{-1}|z|),\\
&\int_{\mathbb{S}^2}x_p x_q\cdot \mathrm{e}^{\sigma c_0^{-1}\hat{x}\cdot z}\mathrm{d}s(\hat{x})=4\pi {\hat z}_p {\hat z}_q\cdot i_2(\sigma c_0^{-1}|z|),\quad p,q=1,2,3,
\end{aligned}
\end{equation}
where  $\hat z:=[\hat z_1, \hat z_2, \hat z_3] =z/|z|$.
Therefore, using the fact that
\begin{align*}
\hat{x}\hat{x}^{\top}=
\begin{bmatrix}
x_1^2 & x_1x_2 & x_1x_3\\
x_1x_2 & x_2^2 & x_2x_3\\
x_1x_3 & x_2x_3 & x_3^2
\end{bmatrix},
\end{align*}
together with formulas \eqref{eq:part},  we derive  that
\begin{equation*}
\int_{\mathbb{S}^2}\left(\mathbb{I}-\hat{x}\hat{x}^{\top}\right) \mathrm{ e}^{\sigma c_0^{-1}  \hat{x}\cdot z} \mathrm{d}s(\hat{x})= \frac{8\pi}{3} i_0(\sigma c_0^{-1} |z|)\mathbb{I}+ \frac{4\pi}{3} (\mathbb{I}-3\hat{z}\hat{z}^{\top})i_2(\sigma c_0^{-1}|z|) .
\end{equation*}

This completes the proof.

\end{proof}

Next, we will introduce another important lemma, which plays a crucial role in analyzing the behaviour of the imaging functional \eqref{find} in the subsequent section.

\begin{lem}\label{lem2}
	Let $\Gamma=\{x\in\mathbb{R}^3\backslash\overline{\Omega}:|x|=r\}$ denote a sphere of radius $r$, and let $\hat q$ be a vector independent of $x$. For  $y\in \Omega$, we define the integral functional
	\begin{align*}\label{crucial}
	V_\omega(z;y):=\int_{\Gamma}\hat \Phi_\omega(x,y) \hat q \, \frac{\mathrm{e}^{-\mathrm{i}\Re(\omega)c_0^{-1}|x-z|}}{4\pi|x-z|}\mathrm{d}s(x), \quad z\in \mathbb{R}^3\backslash \Gamma.
	\end{align*}
	Let $\omega\in\mathbb{C}_{\sigma_0}$  and $C_0>0$ be a constant such that $|\omega|\leq C_0$,  then it holds that
		\begin{equation*}
		\big|V_\omega(z;y)\big|  \leq\frac{ \mathrm{e}^{-\sigma c_0^{-1} r}}{12\pi}\left|2 i_0(\sigma c_0^{-1} |y|)\hat q+  (\mathbb{I}-3\hat{y}\hat{y}^{\top}) \hat q \, i_2(\sigma c_0^{-1}|y|) \right|\bigg\{1+\mathcal{O}\bigg(\frac{|\omega|}{r}\bigg)\bigg\},
		\end{equation*}
as $ r\rightarrow \infty$,
	where the equality holds  when $z=y$. Moreover,  $V_\omega(z;y)$ has the following asymptotic behaviour
	\begin{align*}
	\big|V_\omega(z;y)\big|=C_1 \mathrm{e}^{-\sigma c_0^{-1} r} \mathcal{O}\bigg(\frac{1}{|\Re (\omega)(z-y)|}\bigg)\bigg\{1+\mathcal{O}\bigg(\frac{|\omega|}{r}\bigg)\bigg\},
	\end{align*}
 as $|z-y|\rightarrow \infty$ and $r\rightarrow \infty$, where $C_1 >0$ is a constant.
\end{lem}

\begin{proof}

Using the formula \eqref{fund} and the asymptotic behaviour of the fundamental solution,  we can derive that

\begin{equation}\label{lemcrucial}
\begin{aligned}
V_\omega(z;y)
&=\int_{\Gamma} \left(\mathbb{I}+ \frac{1}{\omega^2\mu_0\epsilon_0}\nabla \nabla^{\top} \right)\frac{\mathrm{e}^{\mathrm{i}\omega c_0^{-1}|x-y|}}{4\pi|x-y|} \hat{q}\,\, \frac{\mathrm{e}^{-\mathrm{ i}\Re(\omega)c_0^{-1}|x-z|}}{4\pi|x-z|}\mathrm{d}s(x)\\
&=\frac{ \mathrm{e}^{-\sigma c_0^{-1} r}}{(4\pi)^2}\int_{\mathbb{S}^2}(\mathbb{I}-\hat{x}\hat{x}^{\top})\hat{q}\,\mathrm{e}^{-\mathrm{i}\Re (\omega)c_0^{-1} \hat x \cdot (z-y)}\mathrm{e}^{\sigma c_0^{-1} \hat x \cdot y}\, \mathrm{d}s(x)\bigg\{1+\mathcal{O}\bigg(\frac{|\omega|}{r}\bigg)\bigg\},
\end{aligned}
\end{equation}
for all $|\omega|\leq C_0$ as $r\rightarrow \infty$.
Furthermore, by lemma $\ref{lem1}$, it follows that
\begin{align*}
\big|V_\omega(z;y)\big|
&\leq \frac{ \mathrm{e}^{-\sigma c_0^{-1} r}}{(4\pi)^2}\bigg|\int_{\mathbb{S}^2}(\mathbb{I}-\hat{x}\hat{x}^{\top}) \hat{q}\,  \mathrm{e}^{\sigma c_0^{-1} \hat x \cdot y}\mathrm{d}s(x)\bigg|\bigg\{1+\mathcal{O}\bigg(\frac{|\omega|}{r}\bigg)\bigg\}\\
&=\frac{ \mathrm{e}^{-\sigma c_0^{-1} r}}{12\pi}\left|2 i_0(\sigma c_0^{-1} |y|)\hat q+  (\mathbb{I}-3\hat{y}\hat{y}^{\top}) \hat q \, i_2(\sigma c_0^{-1}|y|) \right|\bigg\{1+\mathcal{O}\bigg(\frac{|\omega|}{r}\bigg)\bigg\},
\end{align*}
as $r\to \infty$, where the equality holds when $z=y$.

Recalling the asymptotic behaviour of oscillatory integrals \cite[proposition 6, p.344]{Stein1993}, for any fixed point y, it  holds that
\begin{align*}
\bigg|\int_{\mathbb{S}^2}\mathrm{e}^{-\mathrm{ i}\Re (\omega)c_0^{-1} \hat x \cdot (z-y)}\mathrm{e}^{\sigma c_0^{-1} \hat x \cdot y}\mathrm{d}s(x)\bigg|=\mathcal{O}\bigg(\frac{1}{|\Re (\omega)(z-y)|}\bigg),
\end{align*}
 as $|z-y|\rightarrow \infty$.
Given that
\begin{equation*}
0<\big|(\mathbb{I}-\hat{x}\hat{x}^{\top})\hat{q}\big|\leq 2|\hat{q}|,
\end{equation*}
and using formula \eqref{lemcrucial},  there exists a positive constant $C$ such that
\begin{align*}
\big|V_\omega(z;y)\big|=  C_1  \mathrm{e}^{-\sigma c_0^{-1} r} \mathcal{O}\bigg(\frac{1}{|\Re (\omega)(z-y)|}\bigg)\bigg\{1+\mathcal{O}\bigg(\frac{|\omega|}{r}\bigg)\bigg\},
\end{align*}
as $|z-y|\rightarrow \infty$ and $r\to \infty$, which  completes the proof.
\end{proof}


Throughout this work, the unknown object is assumed to consist of several well-separated point-like scatterers, such that
\begin{equation}\label{assum1}
\begin{aligned}
&\Omega:=\bigcup_{j=1}^N\Omega_j,\quad\Omega_j=y_j+\rho B_j,\quad N\in\mathbb{N}.
\end{aligned}
\end{equation}
Here, $\Omega_j$ denotes the $j$-th point-like scatterer, with $y_j$ and $B_j$  representing  the position and normalized shape of the scatterer $\Omega_j$, respectively.  The parameter $\rho>0$ is the scaling factor. Moreover, we define the lower bound of the distances among the separated scatterers as
\begin{align}\label{assum2}
L:=\min_{1\leq i,j\leq N,i\neq j}\mathrm{dist}(\Omega_i,\Omega_j).
\end{align}
where $\mathrm{dist}(\Omega_i,\Omega_j)$ denotes the distance between $\Omega_i$ and $\Omega_j$.

In what follows, we introduce a crucial lemma to show that the scattered field $\hat{E}^s$ is bounded.

	\begin{lem}\label{lem3}
Suppose that  $\omega\in\mathbb{C}_{\sigma_0}$, and that $\hat{E}^i(\cdot,\omega)\in C^1(\Omega)^3$ is divergence free in $\Omega$. Moreover, we assume that  the relative electric permittivity $\epsilon_r:=\epsilon(x)/\epsilon_0\in C^{1,\alpha}(\overline{\Omega})$ of the scattered medium  is lower bounded by $\epsilon_{r,\text{min}}>0$,  and that the assumption \eqref{assum1} holds. Then the solution $\hat{E}^s(\cdot,\omega)$ to \eqref{maxq} is bounded by
		\begin{equation*}
		\begin{aligned}
		\|\hat{E}^s(\cdot,\omega)\|_{C(\Omega)^3}\leq \mathcal{O}(\rho^2|\omega|^2),
		\end{aligned}
		\end{equation*}
		as $\rho|\omega|\rightarrow0$,  where $\rho$ is the scaling parameter of the scatterers and $0<\alpha<1$.
	\end{lem}
\begin{proof}

	According to \cite[Theorem 9.1, p.348]{Colton2019}, the scattered field can be represented as
	\begin{align}\label{es}
	\hat{E}^s(x,\omega)
	&=\omega^2c_0^{-2}\int_{\Omega}(\epsilon_r-1)\hat{G}_\omega(x,y)\hat{E}(y,\omega)\,\mathrm{ d}y-\nabla_x \int_{\Omega}\nabla\cdot\hat{E}(y,\omega)\,\hat{G}_\omega(x,y)\,\mathrm{ d}y\\
	&=\omega^2c_0^{-2}\int_{\Omega}(\epsilon_r-1)\hat{G}_\omega(x,y)\hat{E}(y,\omega)\,\mathrm{ d}y+\nabla_x \int_{\Omega}\frac{\nabla\epsilon_r(y)}{\epsilon_r(y)}\cdot\hat{E}^s(y,\omega)\hat{G}_\omega(x,y)\,\mathrm{ d}y,  \nonumber
	\end{align}
	since  $\hat{E}^i$ is divergence free in $\Omega$ and $\epsilon_r\nabla\cdot\hat{E}^s=-\nabla\epsilon_r \cdot\hat{E}^s$, where
	\begin{equation*}
	\hat G_{\omega}(x,y)=  \frac{\mathrm{e}^{\mathrm{i} \omega c_0^{-1} |x-y|}}{4\pi|x-y|}.
	\end{equation*}

	
	We define the operator $T: C(\Omega)^3\rightarrow C(\Omega)^3$ as
	\begin{align}\label{operatort}
	(TJ)(x):=\omega^2c_0^{-2} \int_{\Omega}\left (\epsilon_r-1\right)   {\hat{G}}_\omega(x,y)J(y)\,\mathrm{d}y,
	\end{align}
	and the operator $R: C(\Omega)^3\rightarrow C(\Omega)^3$ as
	\begin{align}\label{operatorr}
	(RJ)(x):=\nabla_x \int_{\Omega}\frac{\nabla\epsilon_r(y)}{\epsilon_r(y)}\,\hat{G}_\omega(x,y)J(y)\,\mathrm{ d}y, 
	\end{align}
	which indicates  that $(I-R)\hat{E}^s=T\hat{E}$.
	Due to the following relationship
	\begin{align*}
	\left|\int_{\Omega}\frac{\mathrm{ e}^{\mathrm{ i}\omega c_0^{-1}|x-y|}}{4\pi|x-y|^{3-r}}\mathrm{ d}y\right|=\mathcal{O}\left(\rho^r\right),\quad r=1,2,
	\end{align*}
	we can conclude that both the right-hand sides of equations \eqref{operatort} and \eqref{operatorr} are bounded:
\begin{equation*}
  \left\|\omega^2 c_0^{-2} \int_{\Omega}\left (\epsilon_r-1\right)   {\hat{G}}_\omega(x,y)J(y)\,\mathrm{d}y\right\|_{C(\Omega)^3}\leq|\omega|^2c_0^{-2}\epsilon_1 \left\|J\right\|_{C(\Omega)^3}\mathcal{O}(\rho^2),
\end{equation*}
	\begin{equation*}
	\begin{aligned}
	\left\|\nabla_x \int_{\Omega}\frac{\nabla\epsilon_r(y)}{\epsilon_r(y)}\,\hat{G}_\omega(x,y) J(y)\,\mathrm{ d}y\right\|_{C(\Omega)^3}&\leq  C_2 \,\| J\|_{C(\Omega)^3} \int_{\Omega}\left|\frac{\mathrm{i}\omega c_0^{-1}\mathrm{ e}^{\mathrm{ i}\omega c_0^{-1}|x-y|}}{|x-y|}-\frac{\mathrm{ e}^{\mathrm{ i}\omega c_0^{-1}|x-y|}}{|x-y|^2}\right|\mathrm{ d}y\\
	&\leq C_2\,\| J\|_{C(\Omega)^3}\Big\{|\omega|\mathcal{O}(\rho^2)+\mathcal{O}(\rho)\Big\}.
	\end{aligned}
	\end{equation*}
	where $C_2=\|\nabla\epsilon_r\|_{L^\infty(\Omega)}/{\epsilon_{r,\text{min}}}$ and $\epsilon_1=\|\epsilon_r-1\|_{L^\infty(\Omega)}$ is bounded. Thus, we have
	\begin{align*}
	\|(TJ)\|_{C(\Omega)^3}\leq  \|J\|_{C(\Omega)^3}\cdot \mathcal{O}(\rho^2|\omega|^2),
	\end{align*}
	and this implies that
	\begin{align*}
	\|T\|\leq \mathcal{O}(\rho^2|\omega|^2).
	\end{align*}
	By a similar argument, we can get that $\|R\| \leq \mathcal{O}(\rho^2|\omega|+\rho)$. Let $\rho$ and $\omega$ are sufficiently small such that $\|R\|\leq 1/2$. It is easy to verify that $I-R$ is invertible and that $\|(I-R)^{-1}\|\leq 2$. Similarly, we can get that $\Big[I-(I-R)^{-1}T\Big]^{-1}$ is bounded. Hence, it follows that
	
	\begin{equation*}
	\begin{aligned}
	\Big\|\hat{E}^s(\cdot,\omega)\Big\|_{C(\Omega)^3}&=\Big\|(I-R)^{-1}T\hat{E}(\cdot,\omega)\Big\|_{C(\Omega)^3}\\
	&=\Big\|(I-R)^{-1}T \Big[I-(I-R)^{-1}T\Big]^{-1}\hat{E}^i(\cdot,\omega)\Big\|_{C(\Omega)^3}
	\leq \mathcal{O}(\rho^2|\omega|^2).
	\end{aligned}
	\end{equation*}
	as $\rho|\omega|\rightarrow0$, which completes the proof.
\end{proof}


With the three lemmas established above, we will now demonstrate the behavior of the proposed time-domain imaging functional \eqref{ind}.

\begin{thm}\label{main}
	Suppose that the conditions of Lemma \ref{lem2} and \ref{lem3} hold, and that the incident wave satiesfies $\mathcal{E}^i(x,t)\in H_\sigma^{3+s}(\mathbb{R},(H^1(\mathbb{R}^3))^3)$ with $s>0$.   Let $\omega_0:=\xi_{0}+\mathrm{ i}\sigma \in\mathbb{C}_{\sigma_0}$ such that  $\displaystyle{1<|\omega_0|\ll\frac{1}{\rho}}$. If the sampling point $z$ is located in the neighborhood of $\Omega_j$,  then the imaging functional has the following asymptotic behaviour
	\begin{align*}
	\mathcal{I}(z)\leq \rho^6 \mathrm{e}^{-2\sigma  c_0^{-1} r} \bigg\{M_j\bigg(1+\mathcal{O}\Big(\frac{|\omega_0|}{r}\Big)+\mathcal{O}\left(\rho|\omega_{0}|\right)\bigg)^2+\mathcal{O}\Big(\frac{1}{|\omega_{0}|^{2s}}\Big)+N_j\mathcal{O}\Big(\frac{1}{L}\Big)^2 \bigg\},
	\end{align*}
	as $r\rightarrow \infty$, $\rho|\omega_0|\rightarrow 0$, $L\rightarrow \infty$ and $|\omega_{0}|^{2s}\rightarrow \infty$, where the equality holds if and only if $z=y_j$. Here the parameter $M_j$ and $N_j$ are given by
		\begin{align*}
		&M_j=\frac{ c_0^{-4}}{36\pi^3}K_j^2\int_{0+\mathrm{ i}\sigma}^{\xi_0+\mathrm{ i}\sigma}|\omega|^4\bigg|\int_{B_j}(\epsilon_r-1)\hat{E}^i(h,\omega)\,\mathrm{d}h\bigg|^2\mathrm{d}\omega,\\
		&N_j= \frac{c_0^{-3}}{36\pi^3}K_j^2\int_{0+\mathrm{ i}\sigma}^{\xi_0+\mathrm{ i}\sigma}|\omega|^3\bigg|\int_{B_j}(\epsilon_r-1)\hat{E}^i(h,\omega)\, \mathrm{d}h\bigg|^2 \mathrm{d} \omega,
		\end{align*}
	and
	\begin{align*}
	K_j=|i_0(\sigma c_0^{-1}|y_j|)|+\frac{1}{2}(\mathbb{I}-3\hat{y}_j\hat{y}_j^{\top})|i_2(\sigma c_0^{-1}|y_j|)|.
	\end{align*}
	
	Moreover, if the sampling point $z$ is far away from $\Omega$, then it holds that
	\begin{align*}
	\mathcal{I}(z)=\rho^6 \mathrm{e}^{-2\sigma  c_0^{-1} r} \mathcal{O}\left(\frac{1}{\text{dist}(z,\Omega)}\right)\bigg\{\bigg(1+\mathcal{O}\Big(\frac{|\omega_{0}|}{r}\Big)+\mathcal{O}(\rho|\omega_{0}|)\bigg)^2+\mathcal{O}\Big(\frac{1}{|\omega_{0}|^{2s}}\Big) \bigg\},
	\end{align*}
	as $r\rightarrow \infty$, $\rho|\omega_0|\rightarrow 0$, $|\omega_{0}|^{2s}\rightarrow \infty$ and $\mathrm{dist}(z,\Omega)\rightarrow \infty$, where $\mathrm{dist}(z,\Omega):=\min\limits_{1\leq j\leq N}|z-y_j|$. Here the parameter $M'=\max\limits_{1\leq j\leq N}M_j'$ are given by
	\begin{align*}
	M_j'= c_0^{-2}\int_{0+\mathrm{ i}\sigma}^{\xi_{0}+\mathrm{ i}\sigma}|\omega|^2\bigg|\int_{B_j}(\epsilon_r-1)\hat{E}^i(h,\omega)dh\bigg|^2d\omega.\\
	\end{align*}
\end{thm}

\begin{proof}

According to \eqref{eqs} and lemma \ref{lem3}, the scattered field can be represented by
\begin{equation*}
\hat{E}^s(x,\omega)=\omega^2 c_0^{-2}\int_{\Omega}\left (\epsilon_r-1\right)   {\hat{\Phi}}_\omega(x,y)\big(\hat{E}^i(y,\omega)+\mathcal{O}(\rho^2|\omega|^2)\big)\,\mathrm{d}y.
\end{equation*}
For $z\in B(y_j,L/2)$, using the lemma \ref{lem2}, together with the Taylor expansion
\begin{equation*}
\hat{\Phi}_\omega(x,y)=\hat{\Phi}_\omega(x,y_j)+\mathcal{O}\big(\rho|\omega|\big),\quad y\in\Omega_j,\ r=|x|\rightarrow \infty, 
\end{equation*}
one can deduce that
\begin{align*}
&\bigg|\int_{\Gamma}\hat{E}^s(x,\omega)\frac{\mathrm{e}^{-\text{i}\Re(\omega)c_0^{-1}|x-z|}}{4\pi|x-z|}\text{d}s(x)\bigg|\\
&=\bigg|\omega^2c_0^{-2}\rho^3\sum_{j=1}^{N}\bigg\{\int_{\Gamma}\hat{\Phi}_\omega(x,y_j)\frac{e^{-\text{i}\Re(\omega)c_0^{-1}|x-z|}}{4\pi|x-z|}\text{d}s(x)\bigg\}\bigg\{\int_{B_j}(\epsilon_r-1)\hat{E}^i(h,\omega)\text{d}h+\mathcal{O}\big(\rho|\omega|\big)\bigg\}\bigg|\\
&\leq \frac{|\omega|^2c_0^{-1}\rho^3e^{-\sigma c_0^{-1} r}}{6\pi}K_j\bigg|\int_{B_j}(\epsilon_r-1)\hat{E}^i(h,\omega)\text{d}h\bigg|\bigg\{1+\mathcal{O}\Big(\frac{|\omega|}{r}\Big)
+\mathcal{O}\big(\rho|\omega|\big)+\mathcal{O}\Big(\frac{1}{L|\omega|}\Big)\bigg\},
\end{align*}
as $r\to \infty$, $\rho|\omega|\rightarrow 0$, $z\in B(y_j,L/2)$ and $|z-y_j|\geq |y_j-y_i|-|z-y_j|\geq L/2\rightarrow \infty$
where  the equality holds if $z=y_j$ and
\begin{align*}
K_j=|i_0(\sigma c_0^{-1}|y_j|)|+\frac{1}{2}(3\hat{y}_j\hat{y}_j^{\top}-\mathbb{I})|i_2(\sigma c_0^{-1}|y_j|)|.
\end{align*}

On the one hand, we have
\begin{equation}\label{q1}
\begin{aligned}
Q_1(z)&:=\frac{1}{\pi}\int_{0+\mathrm{ i}\sigma}^{\xi_{0}+\mathrm{ i}\sigma}\bigg|\int_{\Gamma} \hat{E}^s(x,\omega)\frac{\mathrm{e}^{-\text{i}\Re(\omega)c_0^{-1}|x-z|}}{4\pi|x-z|}\,\mathrm{d}s(x)\bigg|^2\mathrm{d}\omega\\
&\leq \rho^6 \mathrm{e}^{-2\sigma c_0^{-1} r}\bigg\{M_j\bigg(1+\mathcal{O}\Big(\frac{|\omega_{0}|}{r}\Big)+\mathcal{O}(\rho|\omega_{0}|)\bigg)^2+N_j\mathcal{O}\Big(\frac{1}{L}\Big)^2 \bigg\},
\end{aligned}
\end{equation}
as $r\rightarrow \infty$, $\rho|\omega_0|\rightarrow 0$ and $L\rightarrow \infty$, where the equality holds if $z=y_j$. Here the parameter $M_j$ and $N_j$ are given by
\begin{align*}
&M_j= \frac{c_0^{-4}}{36\pi^3}K_j^2\int_{0+\mathrm{ i}\sigma}^{\xi_{0}+\mathrm{ i}\sigma}|\omega|^4\bigg|\int_{B_j}(\epsilon_r-1)\hat{E}^i(h,\omega)dh\bigg|^2\mathrm{ d}\omega,\\
&N_j= \frac{c_0^{-3}}{36\pi^3}K_j^2\int_{0+\mathrm{ i}\sigma}^{\xi_{0}+\mathrm{ i}\sigma}|\omega|^3\bigg|\int_{B_j}(\epsilon_r-1)\hat{E}^i(h,\omega)dh\bigg|^2\mathrm{ d}\omega.
\end{align*}
It noted that the parameter $M_j$ and $N_j$ are both bounded because $\mathcal{E}^i(x,t)\in H^{3+s}_\sigma(\mathbb{R}, (L^2(\mathbb{R}^3))^3)$ with $s>0$.

On the other hand, by using the Schwartz inequality, it holds that
\begin{equation*}
\begin{aligned}
Q_2(z)&:=\frac{1}{\pi}\int_{\xi_{0}+\mathrm{ i}\sigma}^{+\infty+\mathrm{ i}\sigma}\bigg|\int_{\Gamma} \hat{E}^s(x,\omega)\frac{\mathrm{e}^{-\text{i}\Re(\omega)c_0^{-1}|x-z|}}{4\pi|x-z|}\,\mathrm{d}s(x)\bigg|^2\mathrm{d}\omega\\
&=\frac{1}{\pi}\int_{\xi_{0}+\mathrm{ i}\sigma}^{+\infty+\mathrm{ i}\sigma}\bigg|\omega^2c_0^{-2}\int_{\Omega}\bigg\{\int_{\Gamma}\hat{\Phi}_\omega(x,y)\frac{\mathrm{e}^{-\mathrm{ i}\Re(\omega)c_0^{-1}|x-z|}}{|x-z|}\mathrm{d}s(x)\bigg\}\bigg\{(\epsilon_r(y)-1)\hat{E}(y,\omega)\bigg\}\mathrm{ d}y\bigg|^2\text{d}\omega\\
&\leq\frac{1}{\pi}\int_{\xi_{0}+\mathrm{ i}\sigma}^{+\infty+\mathrm{ i}\sigma}|\omega|^4c_0^{-4}\bigg\{\int_{\Omega}\Big|(\epsilon_r(y)-1)\hat{E}(y,\omega)\Big|^2\text{d}y\bigg\}\bigg\{\int_{\Omega}\Big|\frac{\text{e}^{-\sigma  c_0^{-1} r}}{4\pi}\Big|^2\text{d}y\bigg\}\text{d}\omega\\
&\leq\frac{\text{e}^{-2\sigma  c_0^{-1} r}}{16\pi^3}|\Omega|\,\big\|(\epsilon_r(y)-1)\big\|^2_{L^\infty(\Omega)}\int_{\xi_{0}+\mathrm{ i}\sigma}^{+\infty+\mathrm{ i}\sigma}|\omega|^4c_0^{-4}\int_{\Omega}\big|\hat{E}(y,\omega)\big|^2\text{d}y\,\text{d}\omega.\\
\end{aligned}
\end{equation*}
Due to $\mathcal{E}^i(x,t)\in H^{3+s}_\sigma(\mathbb{R^+},(L^2(\mathbb{R}^3))^3)$ with $s>0$, we have $\mathcal{E}(x,t)\in H^{2+s}_\sigma(\mathbb{R^+},(H^1(\mathbb{R}^3))^3)$. Then there exists a constant $C_1>0$ such that
\begin{align*}
|\omega_{0}|^{2s}\int_{\xi_{0}+\mathrm{ i}\sigma}^{+\infty+\mathrm{ i}\sigma}|\omega|^4\big\|\hat{E}(\cdot,\omega)\big\|^2_{(L^2(\Omega))^3}\text{d}\omega\leq \int_{\xi_{0}+\mathrm{ i}\sigma}^{+\infty+\mathrm{ i}\sigma}|\omega|^{2(2+s)}\big\|\hat{E}(\cdot,\omega)\big\|^2_{(H^1(\Omega))^3}\text{d}\omega<C_3.
\end{align*}
Therefore, it follows that
\begin{align}\label{q2}
Q_2(z)\leq \frac{c_0^{-4}\text{e}^{-2\sigma  c_0^{-1} r}}{16\pi^3}|\Omega|\,\big\|(\epsilon_r(y)-1)\big\|^2_{L^\infty(\Omega)}\frac{C_3}{|\omega_{0}|^{2s}},\quad s>0,
\end{align}
as $|\omega_{0}|^{2s}\rightarrow \infty$ where $|\Omega|$ denotes the volume of the scatterers. Thus,  combining \eqref{q1} and \eqref{q2}, we have
\begin{align*}
\mathcal{I}(z)&=Q_1(z)+Q_2(z)\\
&\leq \rho^6 \mathrm{e}^{-2\sigma  c_0^{-1} r}\bigg\{M_j\bigg(1+\mathcal{O}\Big(\frac{|\omega_{0}|}{r}\Big)+\mathcal{O}(\rho|\omega_{0}|)\bigg)^2+\mathcal{O}\Big(\frac{1}{|\omega_{0}|^{2s}}\Big)+N_j\mathcal{O}\Big(\frac{1}{L}\Big)^2 \bigg\},
\end{align*}
as $r\rightarrow \infty$, $\rho|\omega_0|\rightarrow 0$, $L\rightarrow \infty$ and $|\omega_{0}|^{2s}\rightarrow \infty$.

Additionally, by a similar argument,  if the sampling point $z$ is far from the support of the scatterers $\Omega$, then it holds that
\begin{align*}
\mathcal{I}(z)=\rho^6 \mathrm{e}^{-2\sigma  c_0^{-1} r} \mathcal{O}\left(\frac{1}{\text{dist}(z,\Omega)}\right)\bigg\{\bigg(1+\mathcal{O}\Big(\frac{|\omega_{0}|}{r}\Big)+\mathcal{O}(\rho|\omega_{0}|)\bigg)^2+\mathcal{O}\Big(\frac{1}{|\omega_{0}|^{2s}}\Big) \bigg\},
\end{align*}
as $r\rightarrow \infty$, $\rho|\omega_0|\rightarrow 0$, $|\omega_{0}|^{2s}\rightarrow \infty$ and $\text{dist}(z,\Omega)\rightarrow \infty$ where $\text{dist}(z,\Omega):=\min\limits_{1\leq j\leq N}|z-y_j|$.

This completes the proof.

\end{proof}

\begin{rem}
In the above theorem, we ultilized the truncation frequency $\omega_0$ to complete our proof. It is worthy to mention  that it is necessary to  choose an appropriate $\omega_0$ such that $M_j$ occupied a dominant position in estimation. In practice, we can select $\omega_0$ as the maximum frequency or the center frequency of the incident field.  On the one side, $M_j$ can be interpreted as the integral of the incident wave over the interval
$(0+\mathrm{ i}\sigma, \xi_0+\mathrm{ i}\sigma)$. On the other hand, we need to ensure that
$\mathcal{O}\big(|\omega_0|^{-2s}\big)$ is sufficiently small, which integrates over the interval $(\xi_0+\mathrm{ i}\sigma, +\infty+\mathrm{ i}\sigma)$ associated with the truncated part $Q_2(z)$.
\end{rem}
\begin{rem}
Theorem \ref{main} states that the imaging functional attains local maximum values at the positions of the targets and decreases as the sampling point moves away from the scatterers. Based on this property, we can determine the locations of the unknown point-like scatterers. Furthermore, Theorem \ref{main} allows us to derive several high-order terms in the estimated conclusions.  The term $\mathcal{O}(|\omega_{0}|/r)$  arises from the asymptotic approximation and requires that the frequency parameter
$\omega_0$, associated with the incident field, be bounded. The term $\mathcal{O}(\rho |\omega_{0}|)$ results from quasi-static frequencies, necessitating that the scaling parameter  $\rho\rightarrow 0$,  which means we can only reconstruct small-scale scatterers. The term $\mathcal{O}(|\omega_{0}|^{-2s})$ arises from truncated resonant and high frequencies; we can mitigate this influence by ensuring that $|\omega_{0}|>1$ and $s\gg 1$. In addition, the final term $N_j\mathcal{O}(1/L)$ arises from the multiple scattering among different scatterers $\Omega_j$.
\end{rem}


At the end of this section, we present a stability statement for the proposed time-domain direct sampling method, indicating the robustness of the reconstruction scheme. Due to the limitations of the observation equipment and environmental influences, errors in
$\mathcal{E}^s(x,t)$ are inevitable. We define
 $\mathcal{E}^s_\delta(x,t)|_\Gamma$  as the measured data with noise, where the noise level is denoted by
 $\delta$.

\begin{thm}
Assume that $\mathcal{E}^s(x,t)$ is the solution for \eqref{eq1}, $\mathcal{E}^s_\delta(x,t)$ is the measured scattering field  with $\sup\limits_{t\geq 0,x\in\Gamma}\|\mathcal{E}^s(x,t)-\mathcal{E}^s_\delta(x,t)\|_{\ell^2}\leq\delta$, then we have
\begin{align*}
|\mathcal{I}(z)-\mathcal{I}_\delta(z)|\leq  c\,\delta,\quad z\in\tilde{\Omega},
\end{align*}
where $\mathcal{I}_\delta(z)$ is the indicator function with $\mathcal{E}^s(x,t)$ replaced by $\mathcal{E}^s_\delta(x,t)$ in \eqref{ind} and $c$ is a positive constant independent of sampling point $z$.
\end{thm}
\begin{proof}
Due to the causality of scattered field $\mathcal{E}^s$ and finite propagation speed property of \eqref{ind}, we can truncate the indicator function in $(0,\,T)$ for sufficiently large $T$. Therefore we can derive that
\begin{align*}
|\mathcal{I}(z)-\mathcal{I}_\delta(z)|&=\Bigg|\int_{0}^{T}\biggl|\int_{\Gamma}\mathcal{E}^s(x,t+c_0^{-1}|x-z|)\frac{\mathrm{e}^{-\sigma (t+c_0^{-1}|x-z|)}}{4\pi|x-z|}\mathrm{d}s(x)\biggr|^2\mathrm{d}t\\
&\quad-\int_{0}^{T}\biggl|\int_{\Gamma}\mathcal{E}^s_\delta(x,t+c_0^{-1}|x-z|)\frac{\mathrm{e}^{-\sigma (t+c_0^{-1}|x-z|)}}{4\pi|x-z|}\mathrm{d}s(x)\biggr|^2\mathrm{d}t\Bigg|\\
&\leq\int_{0}^{T}\int_{\Gamma}\bigg(\big|\mathcal{E}^s-\mathcal{E}^s_\delta\big|^2+2\big|\mathcal{E}^s\big|\cdot\big|\mathcal{E}^s-\mathcal{E}^s_\delta\big|\bigg)\frac{\mathrm{e}^{-\sigma (t+c_0^{-1}|x-z|)}}{4\pi|x-z|}\mathrm{d}s(x)\,\mathrm{d}t
  \end{align*}
Since that $\Gamma$ is far from the sampling domain $\tilde{\Omega}$ and the measured data $\mathcal{E}^s_\delta(x,t)$ has limited bound, there exist two constants $c_1,c_2>0$ such that
\begin{align*}
|x-z|\geq c_1,\quad  \big|\mathcal{E}^s\big|_{\ell^2}\leq c_2,\quad x\in\Gamma,\,z\in\tilde{\Omega},
\end{align*}
then we can conclude that
\begin{align*}
|\mathcal{I}(z)-\mathcal{I}_\delta(z)|\leq\frac{T}{4\pi c_1}|\Gamma|(\delta+2c_2)\cdot\delta:=c\,\delta,
\end{align*}
where $|\Gamma|$ denotes the square of the surface $\Gamma$.
\end{proof}

The above theorem demonstrates that the value of indicator function depends continuously on the measurement data, making our method stable with respect to  the noise.

\par

\section{Numerical examples}

In this section, we present several two- and three-dimensional numerical experiments to demonstrate the effectiveness and feasibility of our method.

We assume that  the incident field $\mathcal{E}^i_{\bm{p}}(x,t;y)$  is  generated by a modulated magnetic dipole, acting as a point source located at $y\in\mathbb{R}^3$ with  polarization direction $\bm{p}$.  The modulated casual temporal function is given by $\chi(t)\in C_0^\infty(0,+\infty)$. Thus, the incident field in three dimension emitted from the source located at $y$ satisfies the equation
\begin{align}\label{inc}
\nabla\times(\nabla\times\mathcal{E}^i_{\bm{p}})+\mu_0\epsilon_0\frac{\partial^2\mathcal{E}^i_{\bm{p}}}{\partial t^2}=\chi(t)\cdot\nabla_x\times(\bm{p}\,\delta(x-y)),\quad (x,t)\in\mathbb{R}^3\times\mathbb{R},
\end{align}
where $\delta(x-y)$ is the Dirac delta function. Furthermore, the solution to \eqref{inc} can be explicitly represented by
\begin{equation}\label{incident}
\mathcal{E}^i_{\bm{p}}(x,t;y)=\nabla_x\times(\bm{p}\;G_{\chi}(x,y;t)),
\end{equation}
where
\begin{equation*}
G_{\chi}(x,y;t):=\chi(\cdot)*G(x,y;\cdot)=\frac{\chi(t-c_0^{-1}|x-y|)}{4\pi|x-y|}.
\end{equation*}
Here, $G(x,y;t)$ is the fundamental solution of the time dependent wave equation
\begin{equation*}
G(x,y;t)=\frac{\delta(t-c_0^{-1}|x-y|)}{4\pi|x-y|}.
\end{equation*}
It is noted that  $\nabla\times(\nabla\times\mathcal{E}^i_{\bm{p}})=-\Delta\mathcal{E}^i_{\bm{p}}$  because the incident wave $\mathcal{E}^i_{\bm{p}}$  is divergence free away from the source $y$. By taking the convolution of the source with
$G(x,y;t)$, one can verify that \eqref{incident} is the solution to \eqref{inc}.  As for the incident field generated by modulated magnetic dipoles in TM and TE modes, we will discuss these in detail in section \ref{sectm} and \ref{secte}, respectively.

Synthetic scattering data were generated by finite element method in space and Crank-Nicolson scheme for time discretization.
For the numerical calculations, we truncate the imaging functional in \eqref{ind} by
\begin{equation*}
\begin{aligned}
\mathcal{I}(z)=\int_{0}^{T}\biggl|\int_{\Gamma}\mathcal{E}^s(x,t+c_0^{-1}|x-z|)\frac{\mathrm{e}^{-\sigma (t+ c_0^{-1}|x-z|)}}{4\pi|x-z|}\mathrm{d}s(x)\biggr|^2\mathrm{d}t,
\end{aligned}
\end{equation*}
where $T$ is the terminal time, which depends on the speed of the incident wave and the distance between the scatterers and measurement sensors. Next,  we set the parameters for numerical discretization. Assume that there are $N_s$ receivers $\{x_m|m=1,2,\cdots,N_s\}$,  placed equidistantly along the boundary $\Gamma$. In the time interval $t=0$ to $t=T$, we use a uniform discretization $\{t_n|n=0,1,\cdots,N_t\}$ with a time step $\tau=T/N_t$. The sampling domain $\tilde{\Omega}$, enclosed by the measured surface $\Gamma$ and containing the scatterers $\Omega$, is assumed to be a square in 2D or cube in 3D. It is uniformly discretized using an  $N_h\times N_h$ sampling mesh $\{\mathcal{K}_l|l=1,2,\cdots,N_h\times N_h\}$, with sampling point $z_l\in\mathcal{K}_l$. The indicator function at each sampling point $z_l$ can be computed as follows:
\begin{equation*}
\begin{aligned}
\mathcal{I}(z_l)=\frac{T}{N_t}\sum_{n=1}^{N_t+1}\biggl|\frac{2\pi R}{N_s}\sum_{m=1}^{N_s}\mathcal{E}^s(x_m,t_n+c_0^{-1}|x_m-z_l|)\frac{\mathrm{e}^{-\sigma (t_n+ c_0^{-1}|x_m-z_l|)}}{4\pi|x_m-z_l|}\biggr|^2,
\end{aligned}
\end{equation*}
for $l=1,2,\cdots,N_h\times N_h$ in two dimensions and
\begin{equation*}
\begin{aligned}
\mathcal{I}(z_l)=\frac{T}{N_t}\sum_{n=1}^{N_t+1}\biggl|\frac{4\pi R^2}{N_s}\sum_{m=1}^{N_s}\mathcal{E}^s(x_m,t_n+c_0^{-1}|x_m-z_l|)\frac{\mathrm{e}^{-\sigma (t_n+ c_0^{-1}|x_m-z_l|)}}{4\pi|x_m-z_l|}\biggr|^2,
\end{aligned}
\end{equation*}
in three dimensions.


\subsection{Numerical results for transverse magnetic (TM) mode in 2D}\label{sectm}
In this part, we consider the numerical examples for TM mode, where the electromagnetic scattering problem is reduced to the following system:
\begin{align*}
\nabla\wedge\underline{\nabla}\wedge\mathcal{E}+\mu_0\epsilon(x)\frac{\partial^2\mathcal{E}}{\partial t^2}=0,\quad (x,t)\in\mathbb{R}^2\times\mathbb{R},
\end{align*}
where $\mathcal{E}(x,t)$ is a scalar function and
\begin{align*}
&\underline{\nabla}\wedge\mathcal{E}:=(\partial_y\mathcal{E},-\partial_x\mathcal{E}),\quad \mathcal{E}(x,t)\in\mathbb{R},\\
&\nabla\wedge\mathcal{F}:=\partial_x F_2-\partial_y F_1,\quad \mathcal{F}(x,t)=(F_1,F_2)\in\mathbb{R}^2.
\end{align*}

The background field is generated by a modulated dipole scattered by point-like obstacles. Assuming that the polarization direction $\bm{p}=(p_1,p_2)$, the TM mode of the incident wave $\mathcal{E}^i_{\bm{p}}$ satisfies
\begin{align}\label{tm}
\nabla\wedge\underline{\nabla}\wedge\mathcal{E}^i_{\bm{p}}+\mu_0\epsilon_0\frac{\partial^2\mathcal{E}^i_{\bm{p}}}{\partial t^2}=\chi(t)\cdot\underline{\nabla}\wedge(\bm{p}\,\delta(x-y)),\quad (x,t)\in\mathbb{R}^3\times\mathbb{R},
\end{align}
where $y$ is the location of the source, and the solution of \eqref{tm} can be explicitly represented by
\begin{equation*}
\mathcal{E}^i_{\bm{p}}=\underline{\nabla}\wedge(\bm{p}\,G_\chi)=p_2\frac{\partial G_{\chi}}{\partial x_1}-p_1\frac{\partial G_{\chi}}{\partial x_2},
\end{equation*}
where $x_1$ and $x_2$ denote the two directions of coordinate system in the 2D system.
In the next two examples, we choose $\bm{p}=(0,1)$ and the single source is located at $y=(-8,0)$ in what follows. The observation surface is a circle centered at the origin with radius of $r=6$, and we set $N_s=48$ receivers on this surface. The sampling domain is chosen as a square, i.e. $\tilde{\Omega}=[-2.5,2.5]\times[-2.5,2.5]$ with the  sampling mesh $N_h\times N_h=60\times60$. The scatterers are small rectangles with side length $d=0.2$ located at $(0,1.5)$, $(0,-1.5)$ and $(1.5,0)$. Figure \ref{fig3}(a) exhibits the geometry setting for the receivers by solid dots and the sampling area by dashed line. The scatterers are represented by the blue boxes with relative electric permittivity $\epsilon_r=2$.

\subsubsection{Gaussian-modulated sinusoidal pulse wave}

\begin{figure}[h]
	\centering

	\subfigure[plot of $\chi(t)$]{
		\includegraphics[height=5cm]{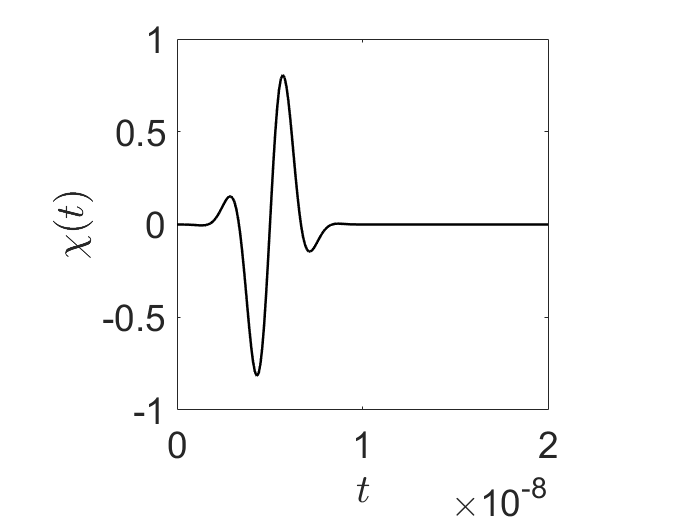}	
	}
	\subfigure[plot of the Fourier spectrum of $\chi(t)$]{
		\includegraphics[height=5cm]{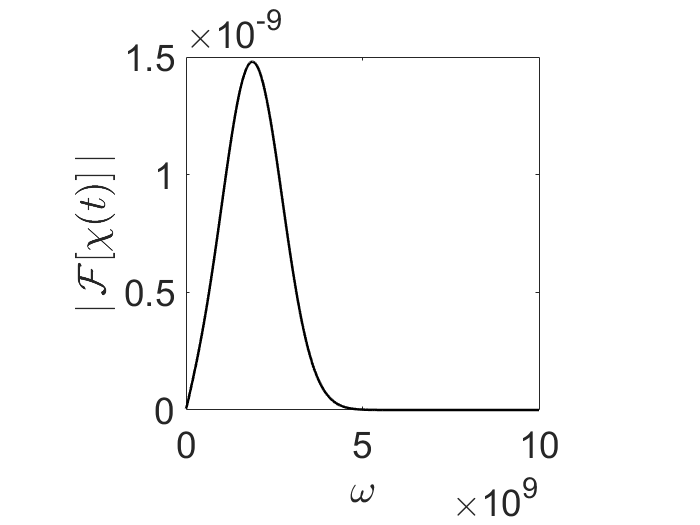}
	}
	\caption{\label{fig2} Plots of the Gaussian modulated  pulse function $\chi(t)$ and  the  corresponding Fourier spectrum with $\lambda=1$ ($\omega\approx1.88\times10^9$).}
\end{figure}

\begin{figure}[ht]
  \centering
  \begin{minipage}[b]{0.99\textwidth}
    \centering
 	\subfigure[geometry setting]{
		\includegraphics[height=5cm]{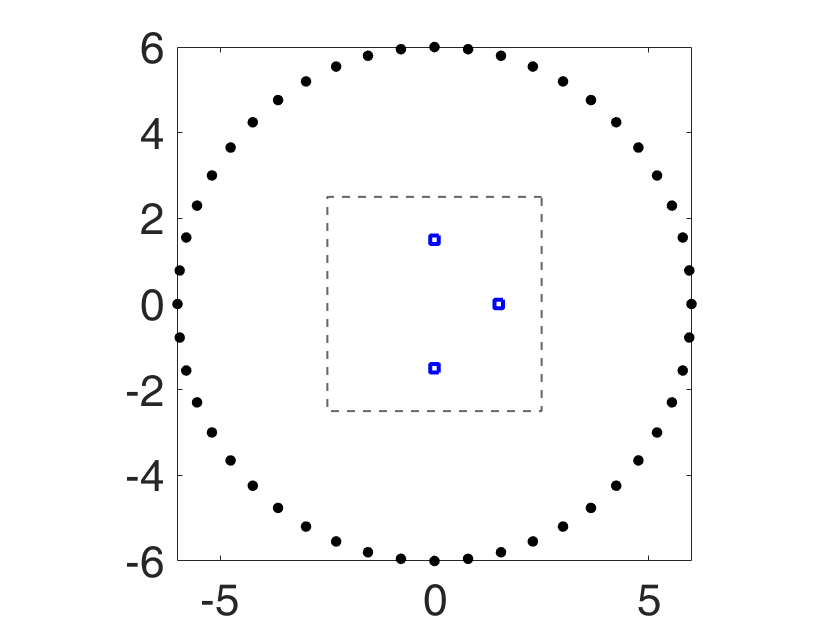}
	}
	\quad
	\subfigure[$\lambda$=2($\omega\approx9.42\times10^8$)]{
		\includegraphics[height=5cm]{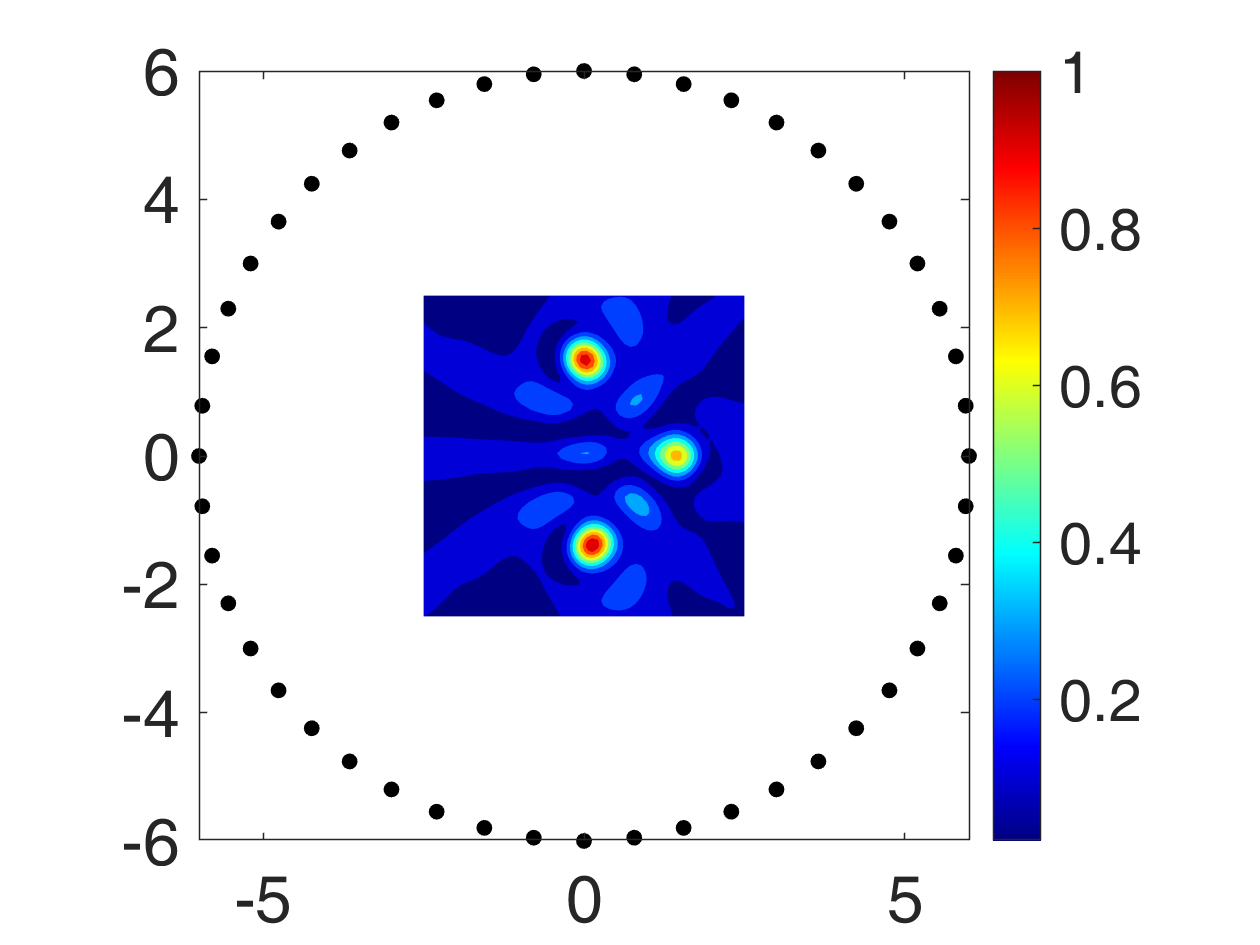}
	}
	\subfigure[$\lambda$=1($\omega\approx1.88\times10^9$)]{
		\includegraphics[height=5cm]{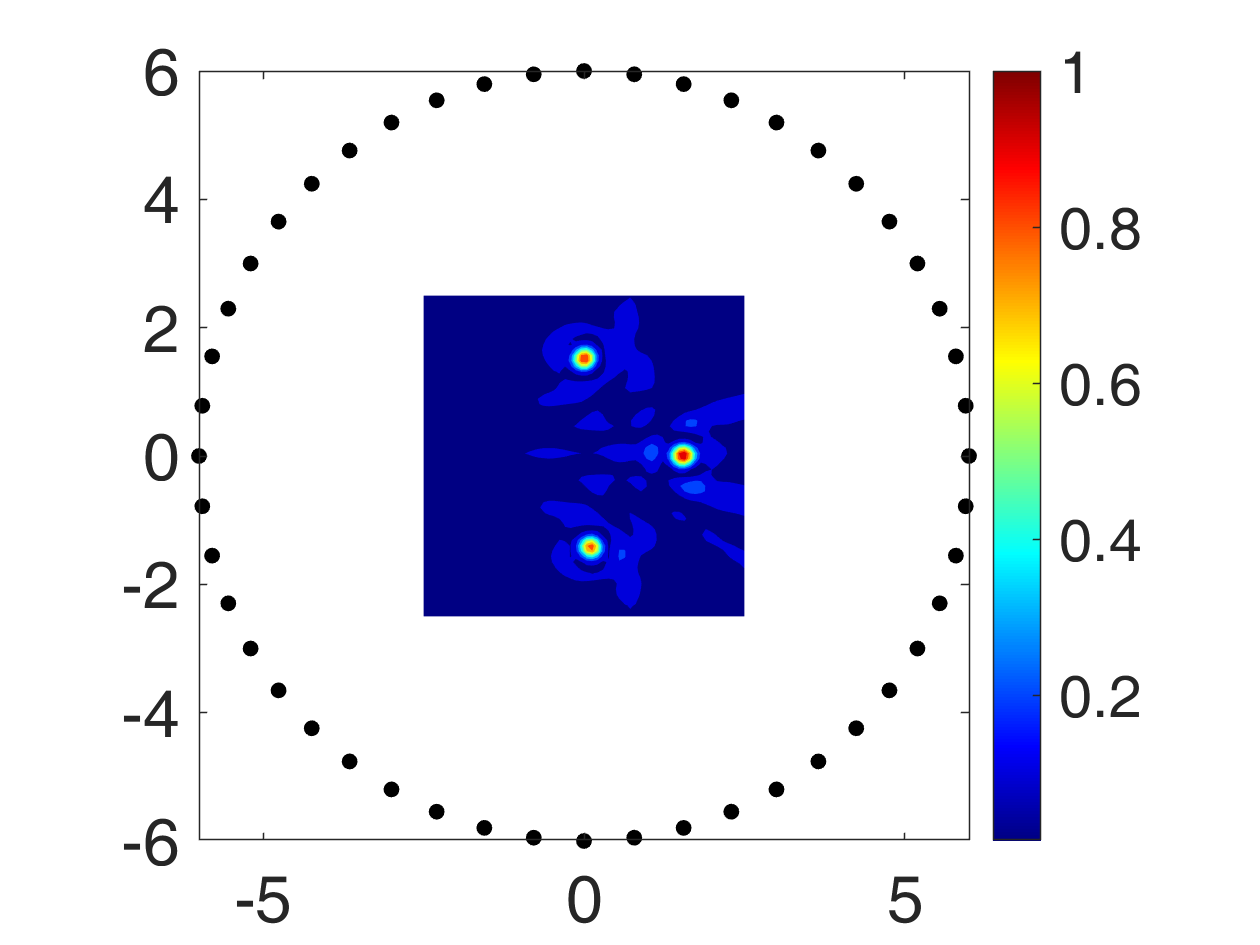}
	}
	\subfigure[$\lambda$=0.5($\omega\approx3.77\times10^9$)]{
		\includegraphics[height=5cm]{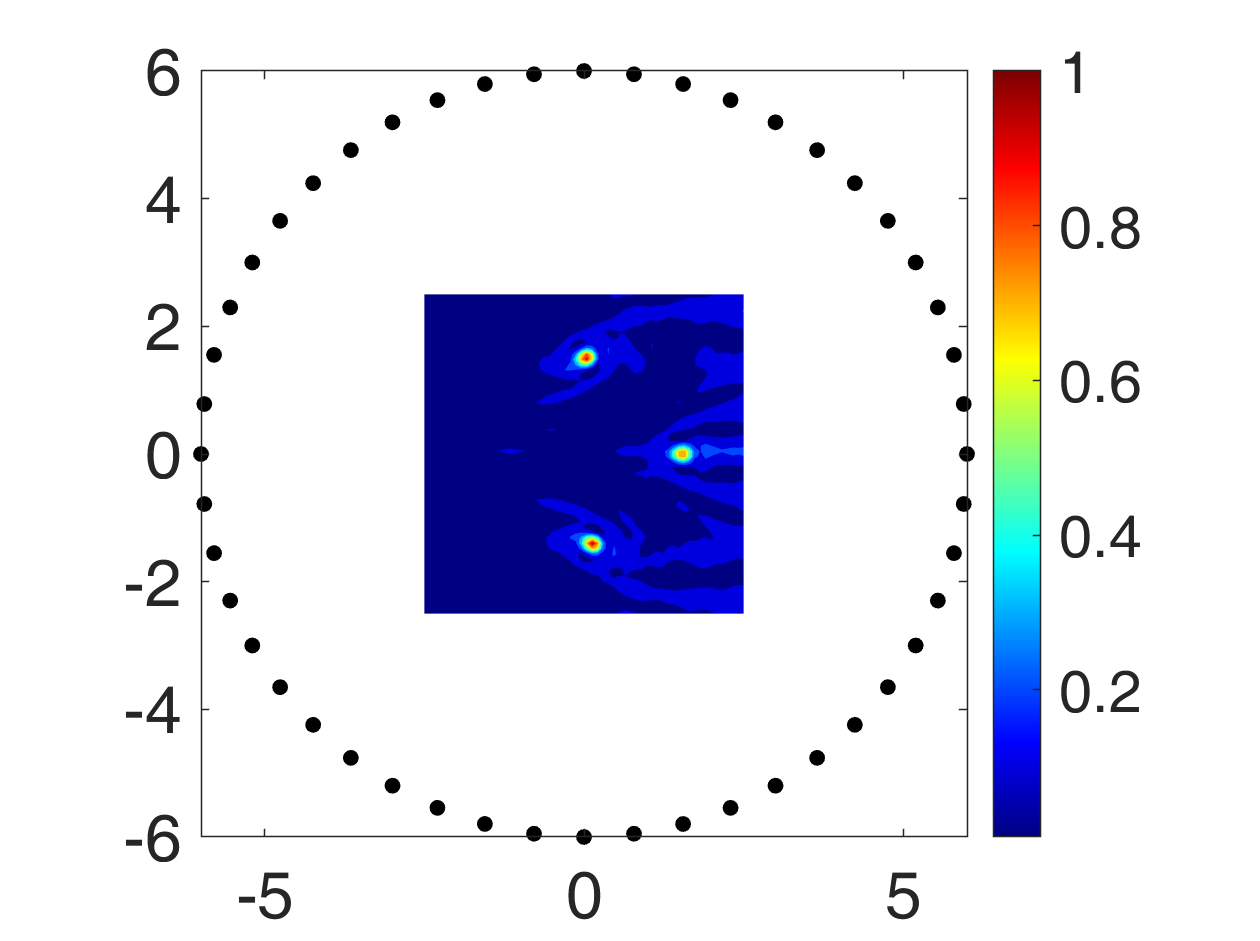}
	}
    \caption{\label{fig3}  Reconstructions of three small rectangular scatterers for the TM mode using the time-domain direct sampling  method $\mathcal{I}(z_l)$   with different wavelengths  $\lambda$. }
  \end{minipage}
  \hfill
  \begin{minipage}[b]{0.99\textwidth}
    \centering
	\subfigure[$\lambda$=1($\omega\approx1.88\times10^9$)]{
		\includegraphics[height=5cm]{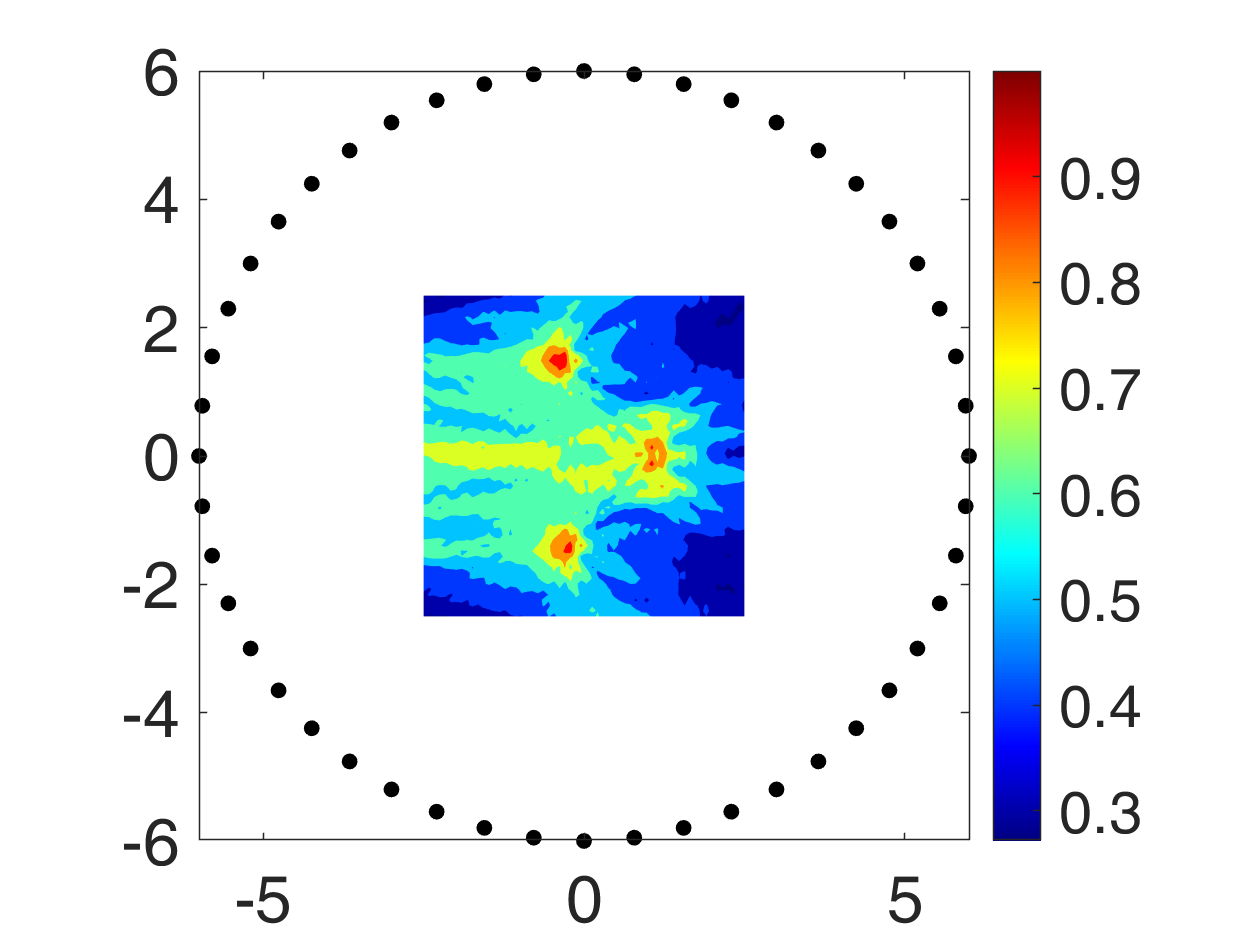}
	}
	\subfigure[$\lambda$=0.5($\omega\approx3.77\times10^9$)]{
		\includegraphics[height=5cm]{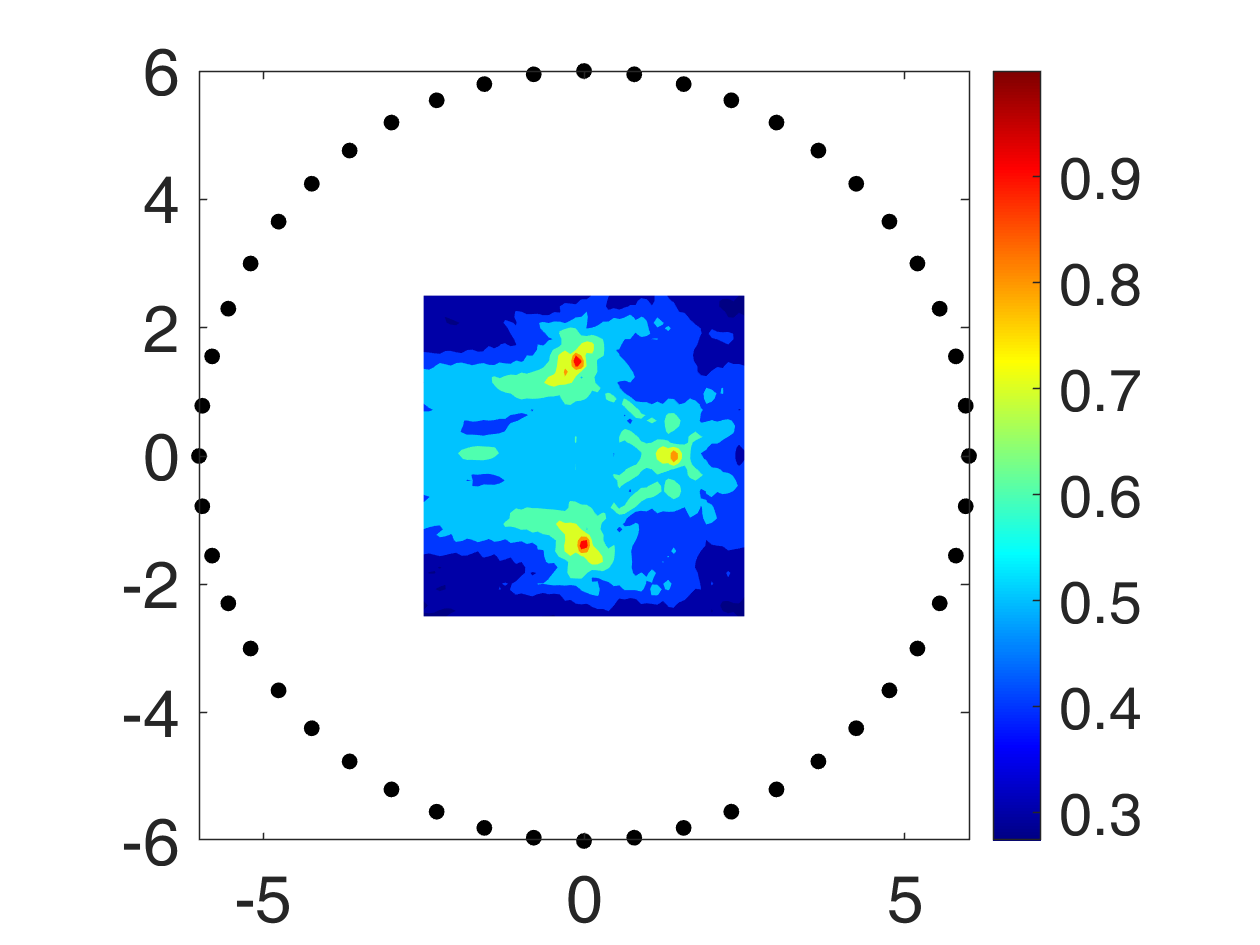}
	}
    \caption{\label{fig8} Reconstructions of three small rectangular scatterers for the TM mode using the total focusing method $\mathcal{I}_{TFM}(z_l)$  with different wavelengths  $\lambda$.}
  \end{minipage}
\end{figure}

In the first example, we use the following  expression as the modulated casual temporal function involved in the incident wave ($\ref{incident}$)
\begin{equation}\label{eq:gaussian}
\chi(t)=
\begin{cases}
\mathrm{e}^{-\frac{(t-t_0)^2}{a^2}}\mathrm{sin}(\omega(t-t_0)),\quad &t\geq 0\\
0 &t<0
\end{cases}
\end{equation}
where $a=1/(2f_0)$, $\omega=2\pi f_0$ and $f_0$ is the peak frequency of the sourse, $t_0$ is the delay time. Figure $\ref{fig2}$(a) shows the graph of the function $\chi(t)$ and figure $\ref{fig2}$(b) displays its wavenumber spectrum, with $f_0=c_0/\lambda$ and $\lambda=1$, where $c_0$ is the speed of light in the vacuum.

Since that the frequency of the incident wave will affect imaging results, we use different value of $\omega$ in $\chi(t)$ for experiments with the same geometry setting. The terminal time $T$ is set to $2\times10^{-7}$s with time step $\tau=2\times 10^{-10}$s. Figure $\ref{fig3}$(b)-(d) show the corresponding contour plots of the indicator functionl $\mathcal{I}(z_l)$ with $\lambda=$2, 1 and 0.5 ($\omega\approx$$9.42\times10^8$, $1.88\times10^9$ and $3.77\times10^9$, respectively).  Obviously this functional get local maximum in the neighbour of the small scale scatterers, and the quality of image gets promoted by increasing the frequency. Above these experiments, we set the parameter $\sigma=0$ to eliminate its effect on the results.

%

To demonstrate the superiority of our proposed approach, we provide a comparison between the time-domain direct sampling method and the total focusing method using the same measurement data. The imaging functional for the total focusing method with a single point source is given by
\begin{equation*}
\mathcal{I}_{TFM}(z):=\bigg|\int_{\Gamma}\mathcal{E}^s(x,t_0+c_0^{-1}|x-z|+c_0^{-1}|y-z|)\mathrm{ d}x\bigg|,
\end{equation*}
where $t_0$ is the peak time of $\chi$ and $y$ is the location of the point source. Figure \ref{fig8}(a)-(b) show the corresponding contour plots of the imaging functional $\mathcal{I}_{TFM}$ with different frequency parameters of the incident field. By comparing Figures \ref{fig3}(c)-(d) with Figures \ref{fig8}(a)-(b), we observe that our proposed time-domain direct sampling method produces clearer imaging results.

To test the stability of the proposed method, random noise is added to the synthetic scattered data $\mathcal{E}^s(x,t)$ as follows
\begin{align*}
\mathcal{E}^s_{\delta}(x_m,t_n):=\mathcal{E}^s(x_m,t_n)+\delta R_{m,n}\frac{\mathop{\mathrm{max}}\limits_{m,n}|\mathcal{E}^s(x_m,t_n)|}{|\mathcal{E}^s(x_m,t_n)|}\mathcal{E}^s(x_m,t_n),
\end{align*}
where $R_{m,n}$ is a random number that obey the standard Gaussian distribution and $\delta$ denotes the noise level.  Figure \ref{fig9} present the contour plots of the imaging functional $\mathcal{I}_{\delta}(z_l)$ of scattered data $\mathcal{E}^s_{\delta}(x_m,t_l)$ with different noise level $\delta$. The imaging plots show that our proposed time-domain direct sampling method is robust to noise and can tolerate high noise levels.

\begin{figure}[h]
	\centering

    \subfigure[$\delta=0.1$]{
    	\includegraphics[height=5cm]{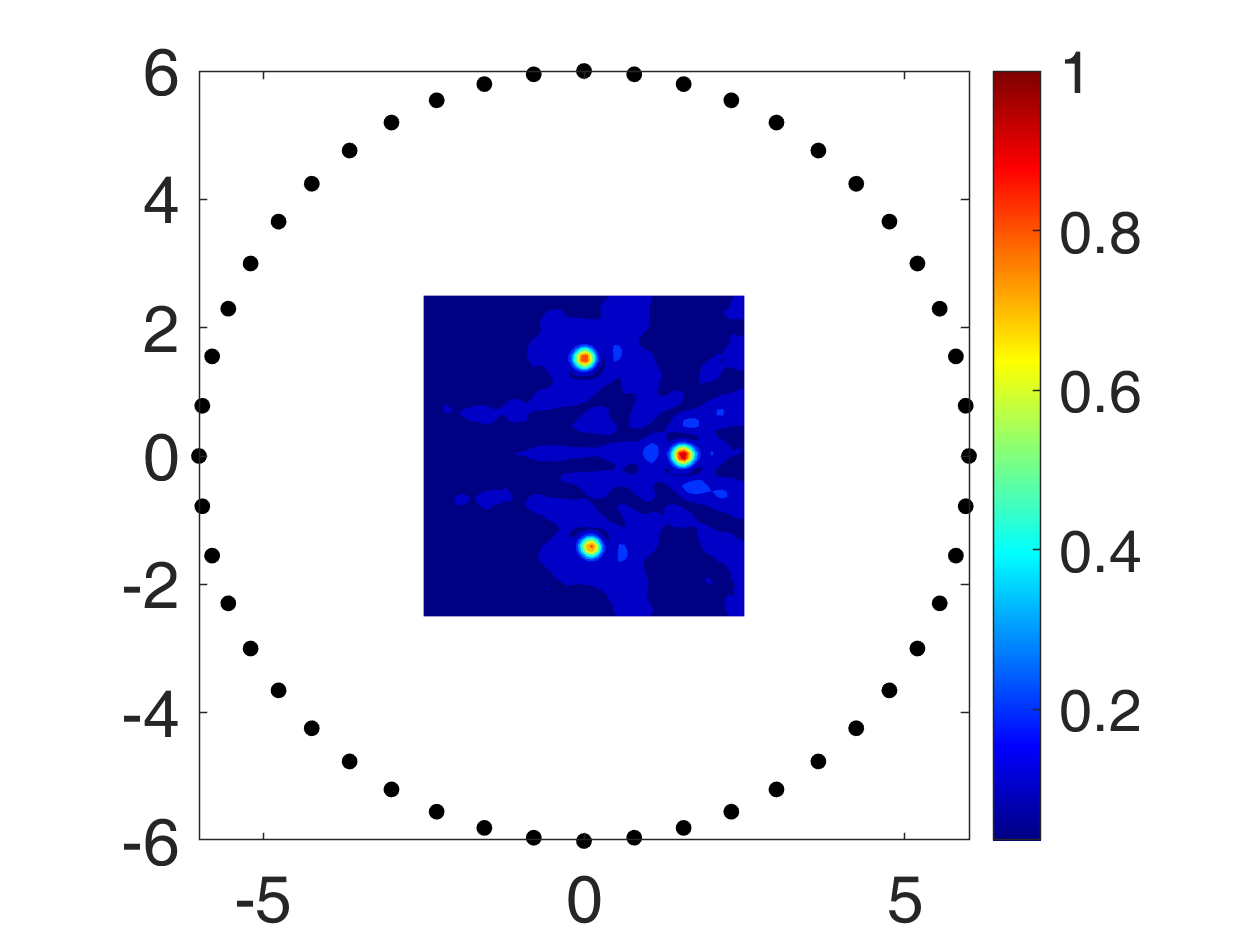}
    }
        \subfigure[$\delta=0.2$]{
    	\includegraphics[height=5cm]{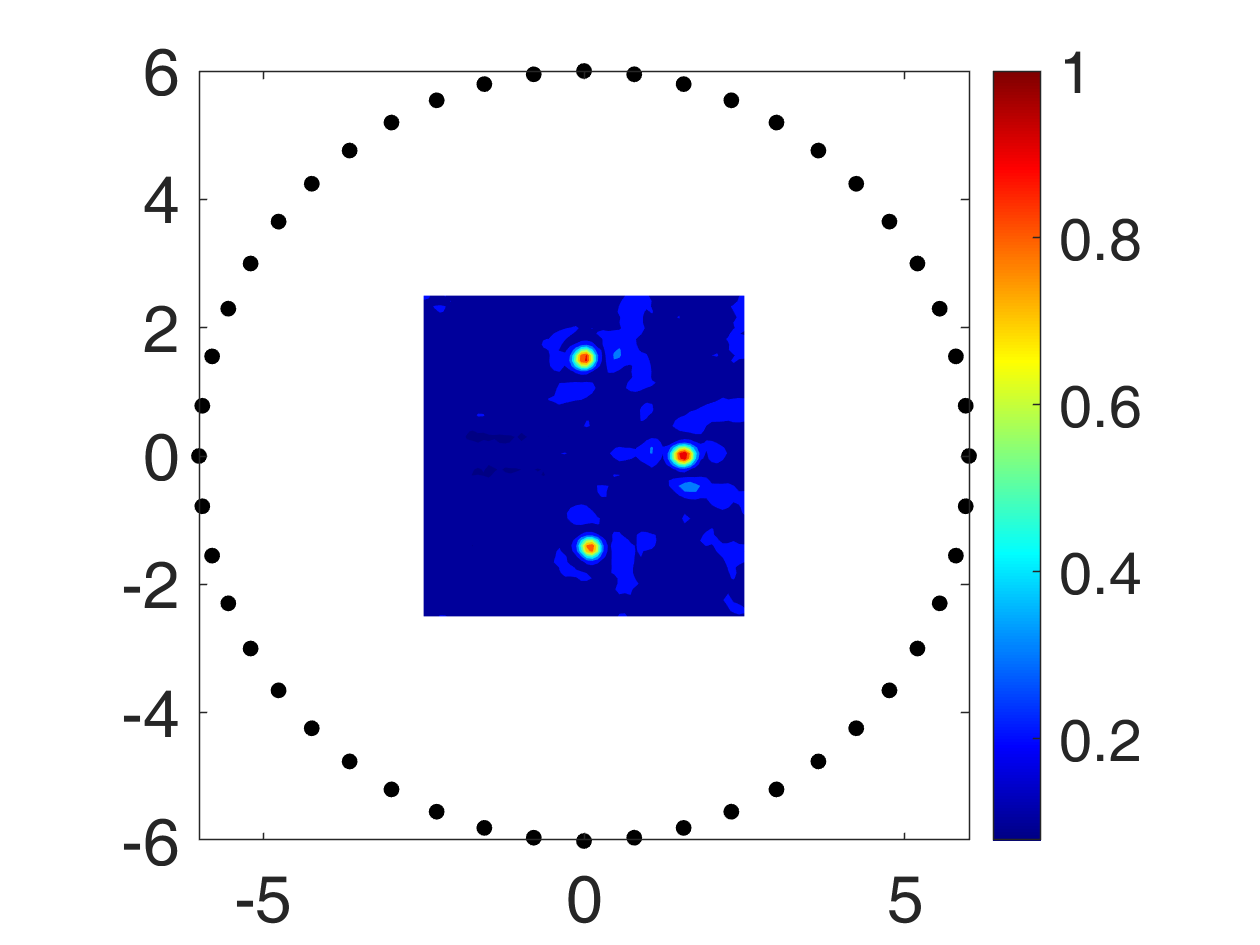}
    }
    \subfigure[$\delta=0.4$]{
	    \includegraphics[height=5cm]{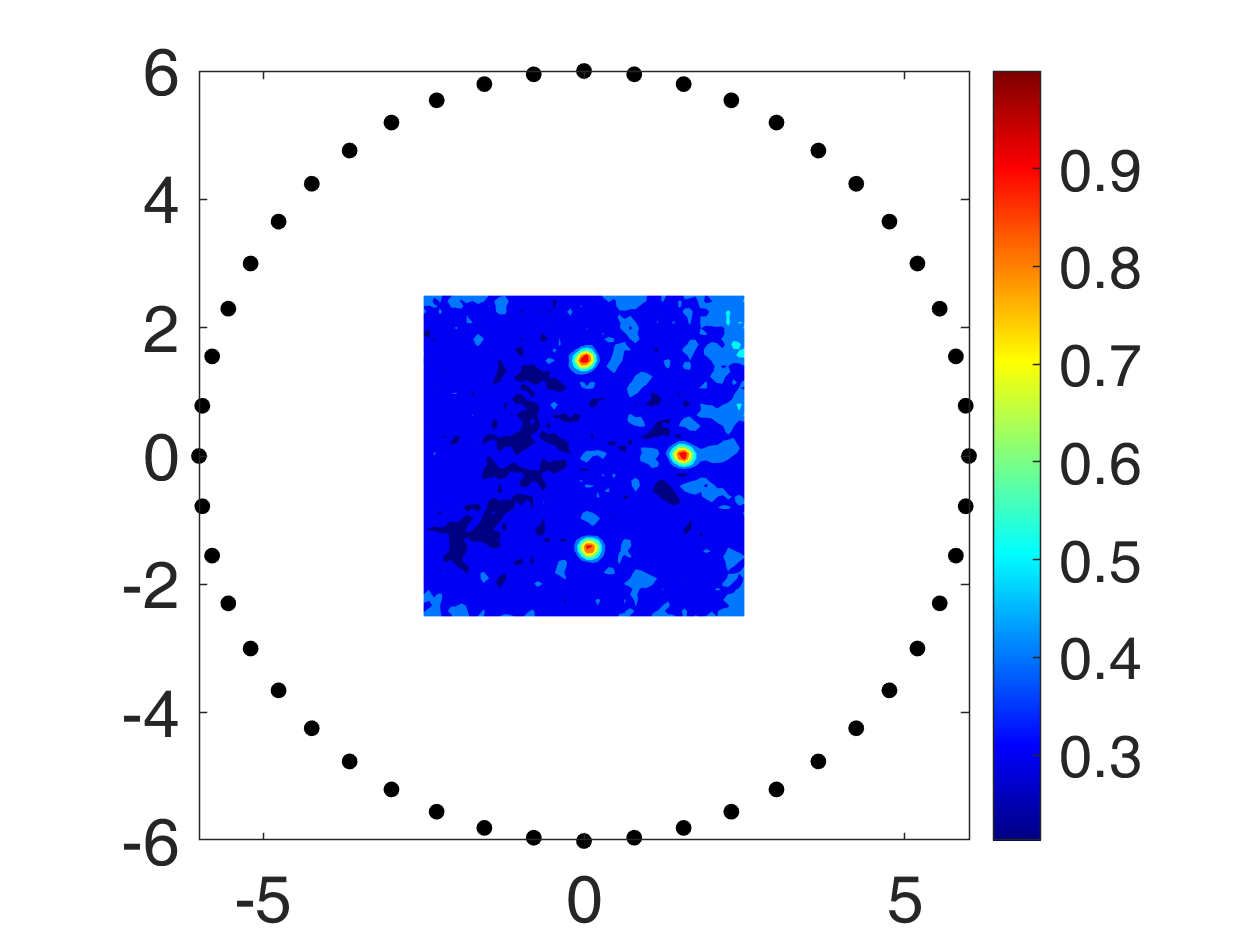}
    }
    \subfigure[$\delta=0.6$]{
    	\includegraphics[height=5cm]{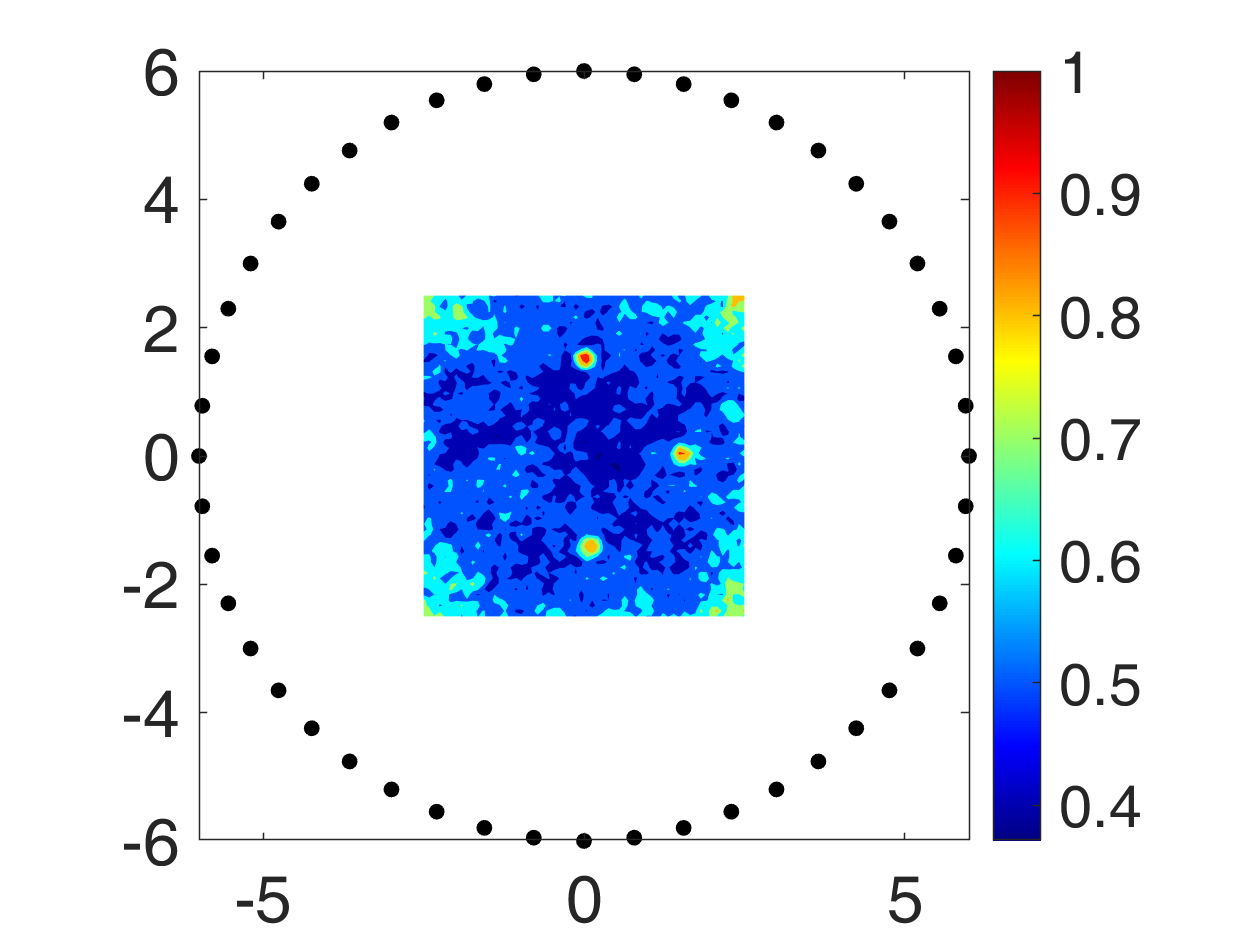}
    }
	\caption{\label{fig9} Contour plots of the reconstructions for the TM mode  with different noise level $\delta$.}
\end{figure}

\subsubsection{Smooth sawtooth wave}
In the second example, we utilitze the sawtooth wave as the modulated $\chi(t)$ function associated with the background field excited by a magnetic dipole
\begin{equation*}
\chi(t)=
\begin{cases}
\displaystyle{\int_{-\infty}^{+\infty}\left(\frac{b\tau+\pi}{2\pi}-\left\lfloor\frac{b\tau+\pi}{2\pi}\right\rfloor-\frac{1}{2}\right)\text{e}^{-c(t-\tau)^2}\mathrm{ d}\tau},\quad &t\geq 0\\
0, &t<0
\end{cases}
\end{equation*}
where $b=\pi f_0$, $f_0$ is the frequency of the signal and $c$ is the smoothing parameter. Figure \ref{fig4}(a) represents the graph of the function $\chi(t)$ with $f_0=c_0/\lambda$ and $\lambda=1$, where $c_0$ is the speed of the light in the vacuum.
\par
In this part, we try different terminal time $T$ for the experiments with same geometry setting. Figure \ref{fig4}(b)-(d) show the corresponding contour plots of the indicator function $\mathcal{I}(z_l)$ with $T=4\times10^{-8}$, $T=6\times10^{-8}$ and $T=10^{-7}$. These figures show that the reconstruction becomes more accurate as the terminal time increases, provided
$T$ is sufficiently large.

\begin{figure}[h]
	\centering
	\subfigure[geometry setting]{
		\includegraphics[height=5cm]{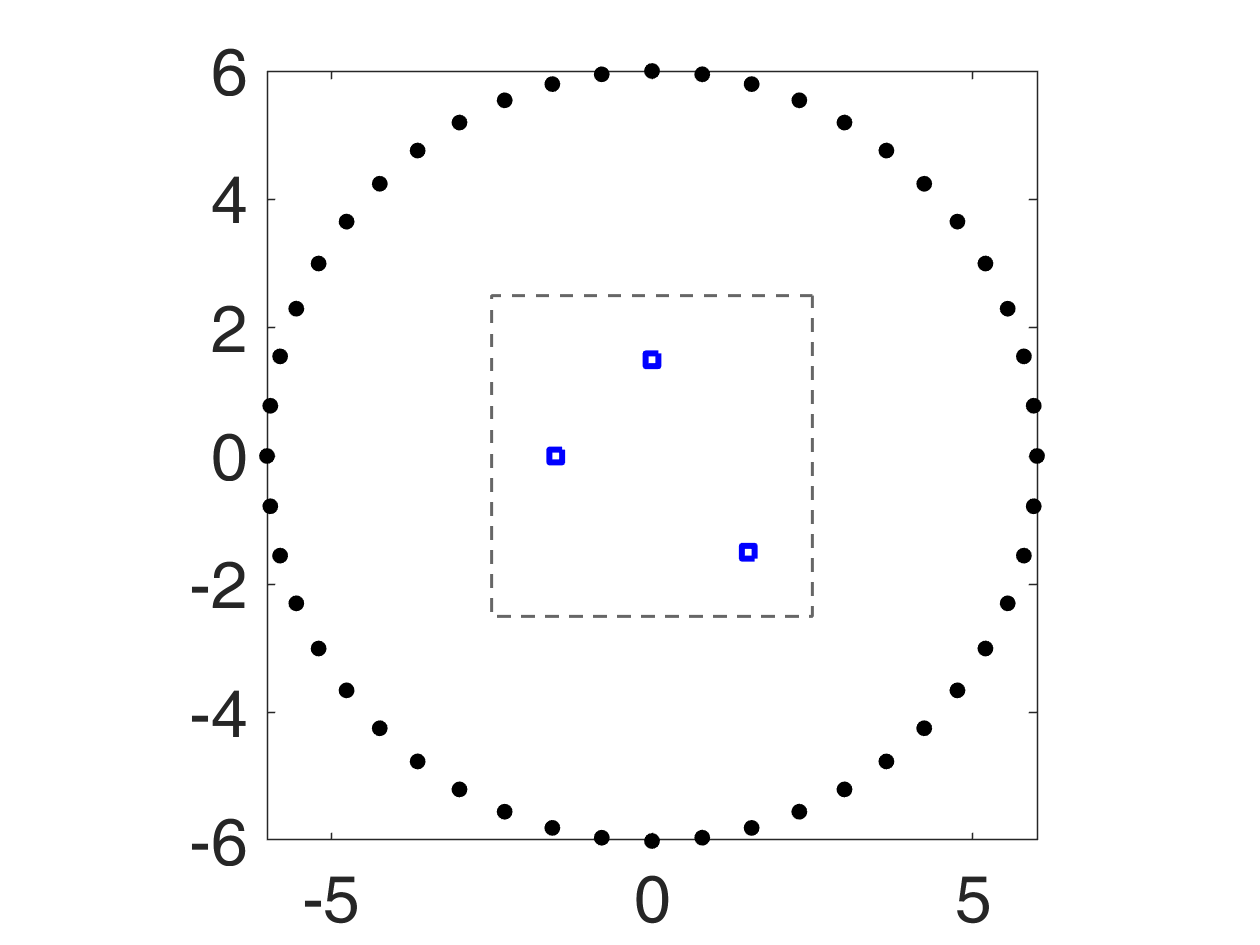}}
	\subfigure[$T$=$3.2\times10^{-8}$]{
		\includegraphics[height=5cm]{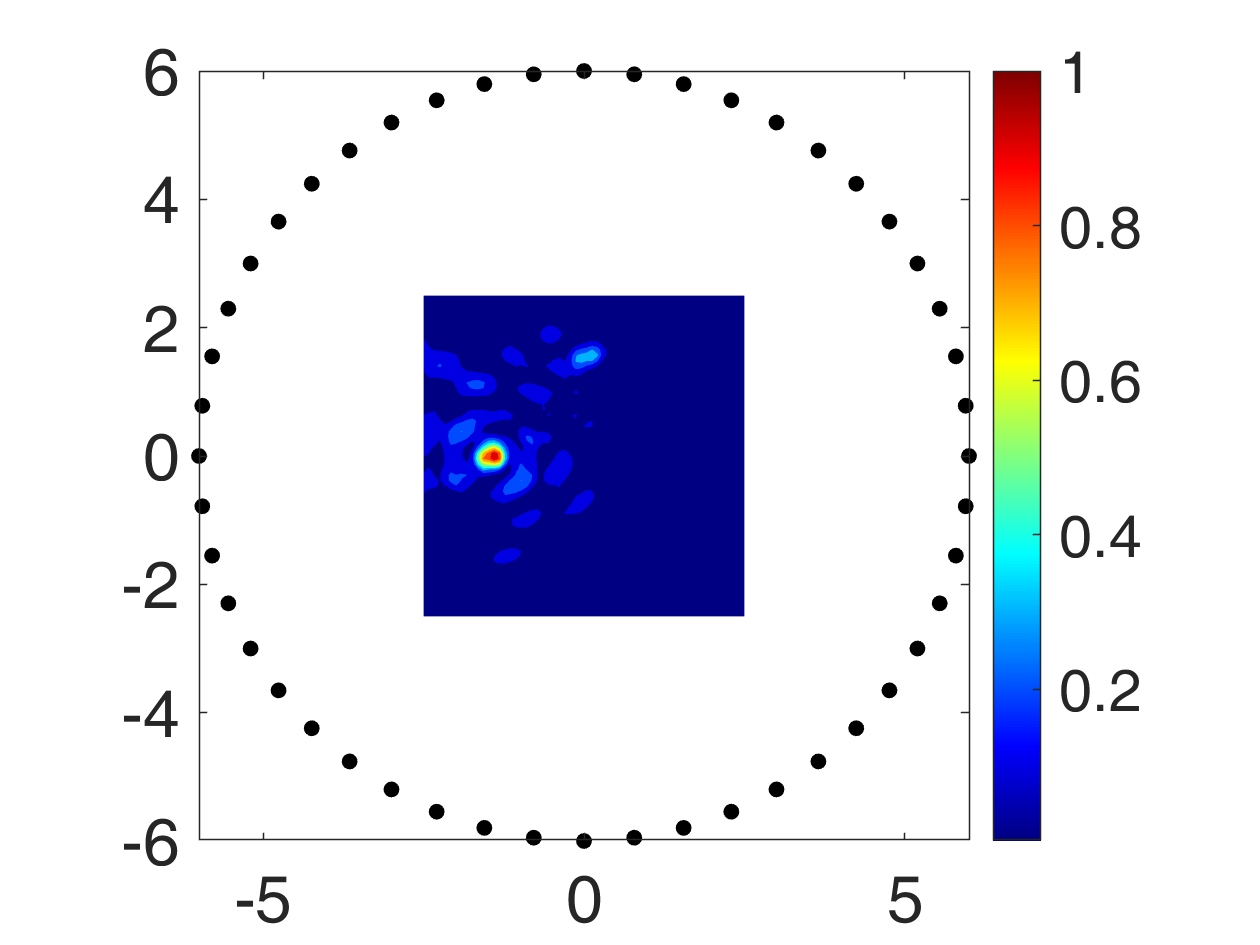}}
	\subfigure[$T$=$4\times10^{-8}$]{
		\includegraphics[height=5cm]{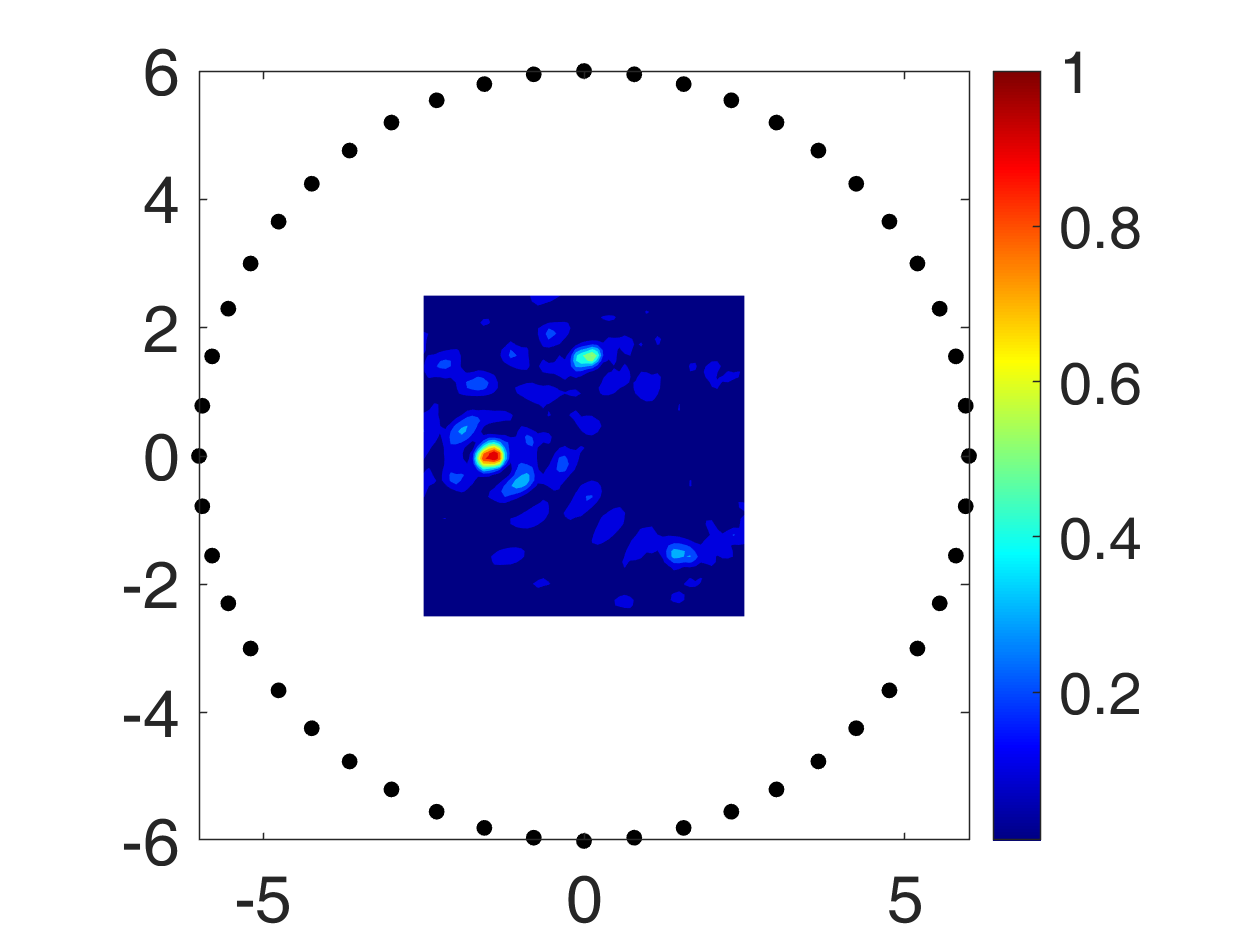}}
	\subfigure[$T$=$8\times 10^{-8}$]{
		\includegraphics[height=5cm]{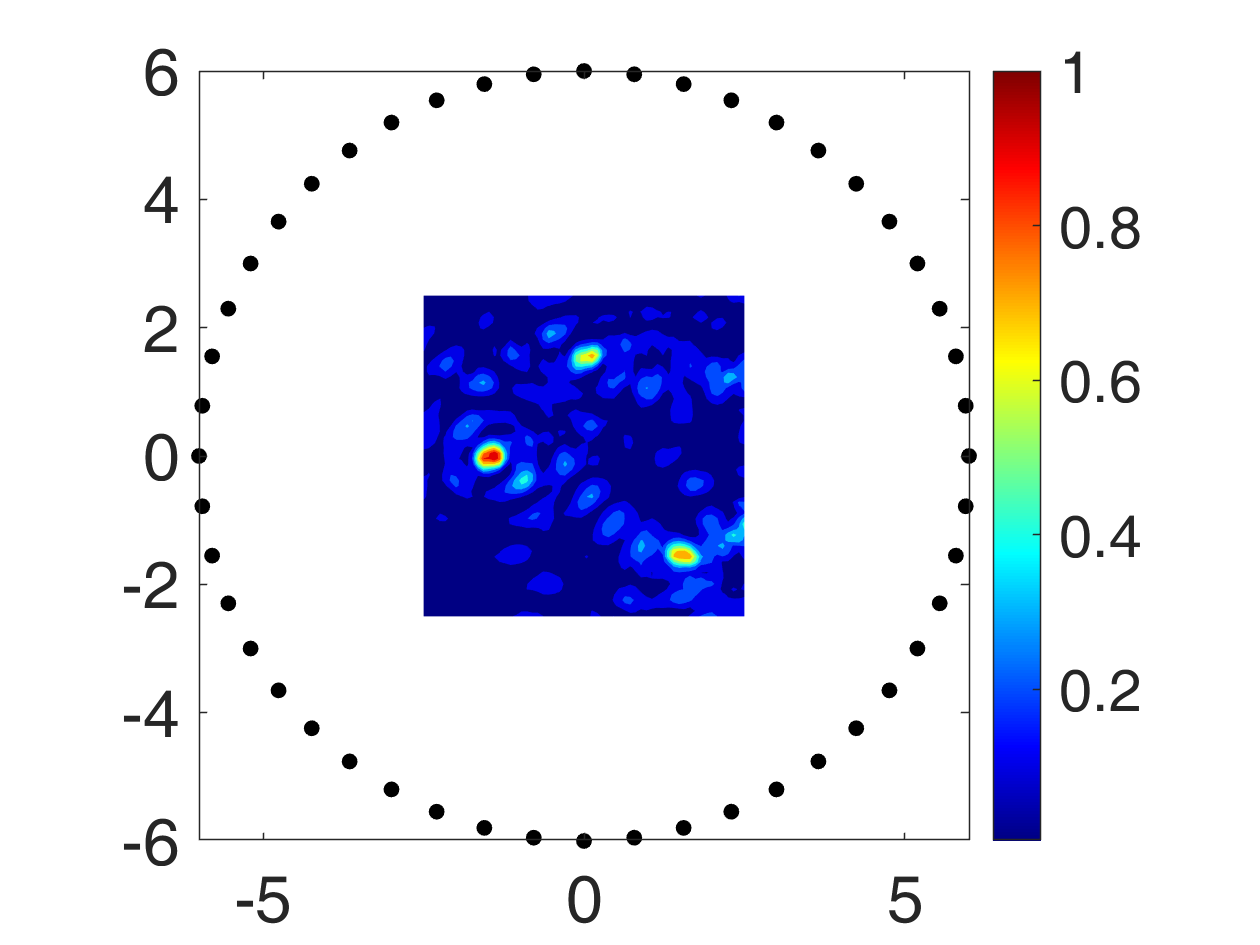}}
	\caption{\label{fig4}
 Contour plots of the reconstructions for  the TM mode with different terminal time $T$.}
\end{figure}

\subsection{Numerical results for transverse electric (TE) mode in 2D}\label{secte}
Next, we consider a numerical example for the TE mode. The electromagnetic scattering problem for the TE mode is reduced to the following system:
\begin{align*}
\underline{\nabla}\wedge\nabla\wedge\mathcal{E}+\mu_0\epsilon(x)\frac{\partial^2\mathcal{E}}{\partial t^2}=0,\quad (x,t)\in\mathbb{R}^2\times\mathbb{R},
\end{align*}
where $\mathcal{E}(x,t)=(E_1(x,t),E_2(x,t))$.
Similarly to the previous examples, we assume the scatterers to be point-like and the incident wave is generated by a modulated magnetic dipole.
The TE mode of the incident wave $\mathcal{E}^i$ satisfying
\begin{align}\label{te}
\underline{\nabla}\wedge\nabla\wedge\mathcal{E}^i+\mu_0\epsilon_0\frac{\partial^2\mathcal{E}^i}{\partial t^2}=\chi(t)\cdot\nabla\wedge\delta(x-y),\quad (x,t)\in\mathbb{R}^3\times\mathbb{R},
\end{align}
where $y$ is the location of the source, and we can explicitly represent the solution of \eqref{te}
\begin{equation*}
\mathcal{E}^i=\nabla\wedge G_\chi
=\left(\frac{\partial G_{\chi}}{\partial x_2},\,-\frac{\partial G_{\chi}}{\partial x_1}\right)^{\top},
\end{equation*}
where $x_1$ and $x_2$ denote the two directions of coordinate system in the 2D system.
In the third example, we continue to use the Gaussian-modulated sinusoidal pulse wave defined in \eqref{eq:gaussian} as the modulated casual function $\chi(t)$.
The geometry setting of the observation points and the sampling area is the same as in the previous experiment. The scatterers are small rectangles with the side length $d=0.2$  located at $(0,1.5)$ and $(0,-1.5)$. The geometry setting of the observation points, sampling area, and scatterers are illustrated in figure \ref{fig5}(a).

Since that partial aperture imaging has garnered significant attention from researchers, we present reconstructions with different observation apertures. Figure \ref{fig5}(b) shows the contour plots of the imaging functional $\mathcal{I}(z_l)$ with a full aperture, while figure \ref{fig5}(c)-(d) display reconstructions with observation aperture $\displaystyle{\theta\in\left(\frac{\pi}{2},\frac{3\pi}{2}\right)}$ and $\displaystyle{\theta\in\left(\frac{3\pi}{4},\frac{5\pi}{4}\right)}$, respectively.


\begin{figure}[h]
	\centering
	
	\subfigure[geometry setting]{
		\includegraphics[height=5cm]{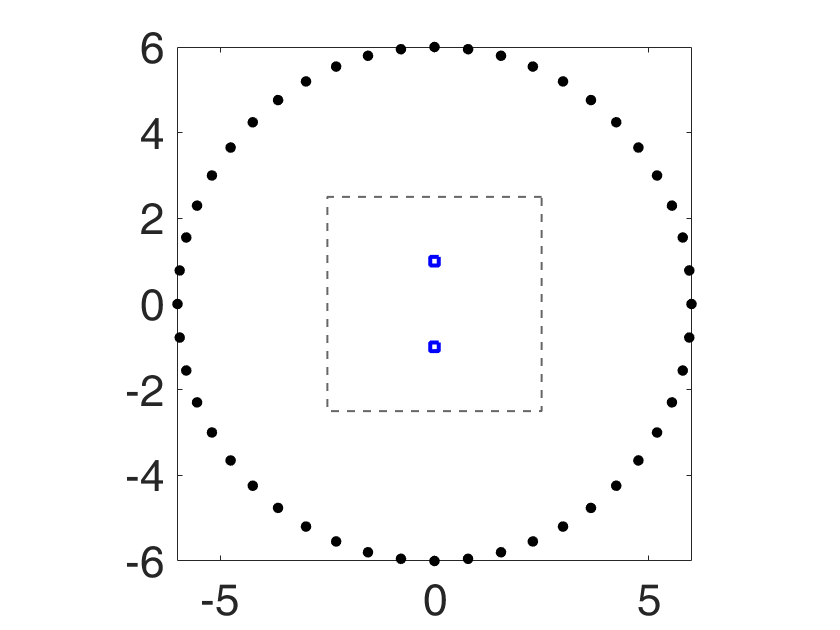}
	}
	\quad
	\subfigure[$\displaystyle{\theta\in\left(0, \, 2\pi \right)}$]{
		\includegraphics[height=5cm]{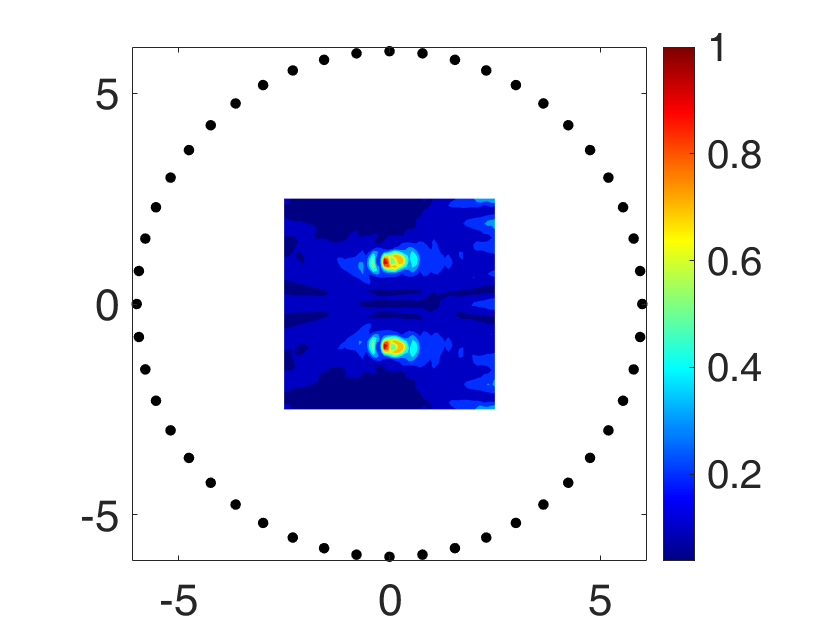}
	}
	\quad
	\subfigure[$\displaystyle{\theta\in\left(\frac{\pi}{2},\frac{3\pi}{2}\right)}$]{
		\includegraphics[height=5cm]{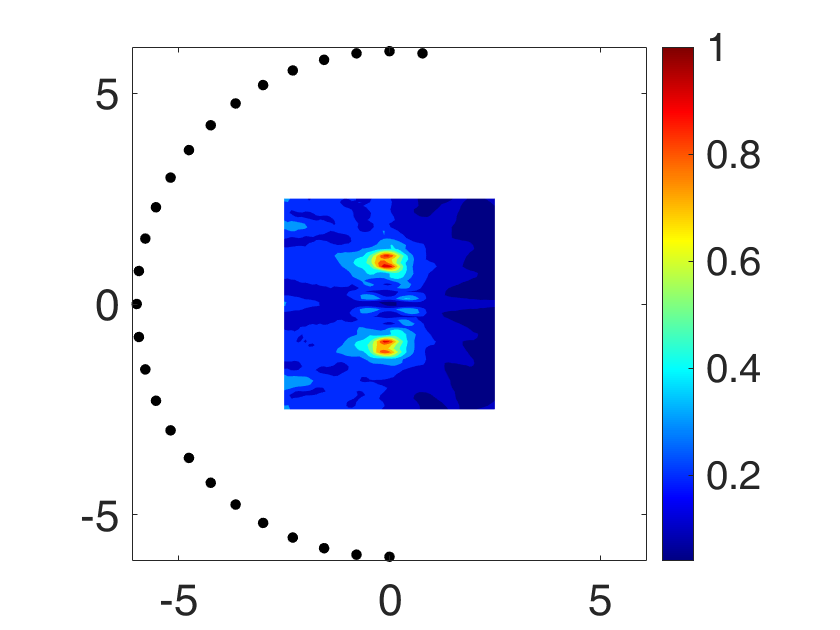}
	}
	\quad
	\subfigure[$\displaystyle{\theta\in\left(\frac{3\pi}{4},\frac{5\pi}{4}\right)}$]{
		\includegraphics[height=5cm]{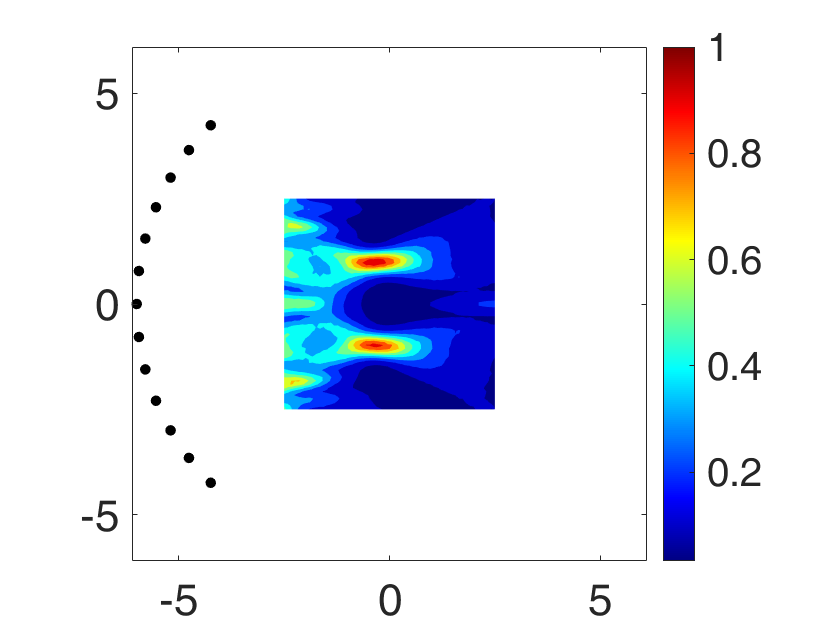}
	}
	\caption{\label{fig5} Contour plots of the reconstructions for  the TE mode with different observation apertures $\theta$. }
\end{figure}

\subsection{Numerical results in 3D}
Finally, we consider a numerical example in three dimensions. The incident field given by \eqref{incident} generated by a modulated dipole scattered at $(0,-8,0)$, which is represented by
\begin{align*}
\mathcal{E}^i_{\bm{p}}=\left(0,\frac{\partial G_\chi(x_1,x_2+8,x_3;t)}{\partial x_3},-\frac{\partial G_\chi(x_1,x_2+8,x_3;t)}{\partial x_2}\right).
\end{align*}
Here $x_1$, $x_2$ and $x_3$ denote the three directions of coordinate system in 3D.
We continue to use the Gaussian-modulated pulse in \eqref{eq:gaussian} as the modulated casual function $\chi(t)$. We first consider a single cube with side length $d=1$ centered at the origin, as the scatterer. Receivers are uniformly distributed on the surface of the cube with side length $R=6$, and the number of the receivers on each face is $49$. The sampling domain is chosen as $\tilde{\Omega}=[-2,2]\times[-2,2]\times[-2,2]$, and the sampling mesh consists $N_h\times N_h\times N_h=30\times30\times30$. In the next 3D example, we reconstruct disconnected scatterers, consisting of two cubes with a side length $d=0.5$.

Figure \ref{fig6}(a) illustrates the geometry of the observation surface (dashed line) and the scatterer cube (solid line). Figure \ref{fig6}(b) shows the locations of the receivers marked by asterisks; for clarity, only the receivers on one surface are depicted. Figures \ref{fig6}(c)-(d) respectively present the slices  and  isosurface plots of the imaging functional
$\mathcal{I}(z)$. Similarly, figure \ref{fig7}(a) illustrates the geometry for two disconnected cubes, and figures \ref{fig7}(b)-(c) show the reconstructed results based on $\mathcal{I}(z)$ and the corresponding slices plot.

\begin{figure}[h]
	\centering
	
	\subfigure[geometry setting]{
		\includegraphics[height=5cm]{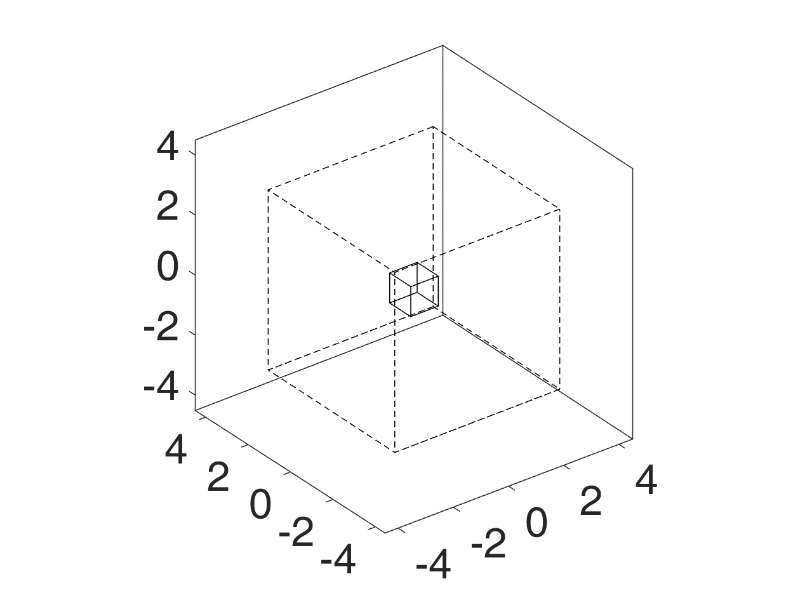}
	}
	\quad
	\subfigure[plot of the receivers on one surface]{
		\includegraphics[height=5cm]{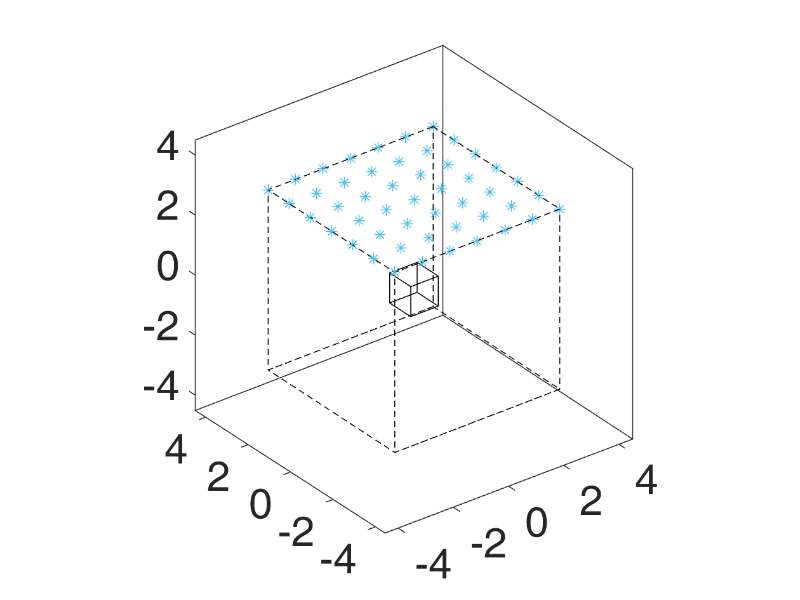}
	}
	\quad
	\subfigure[slices plot of $\mathcal{I}(z)$]{
		\includegraphics[height=5cm]{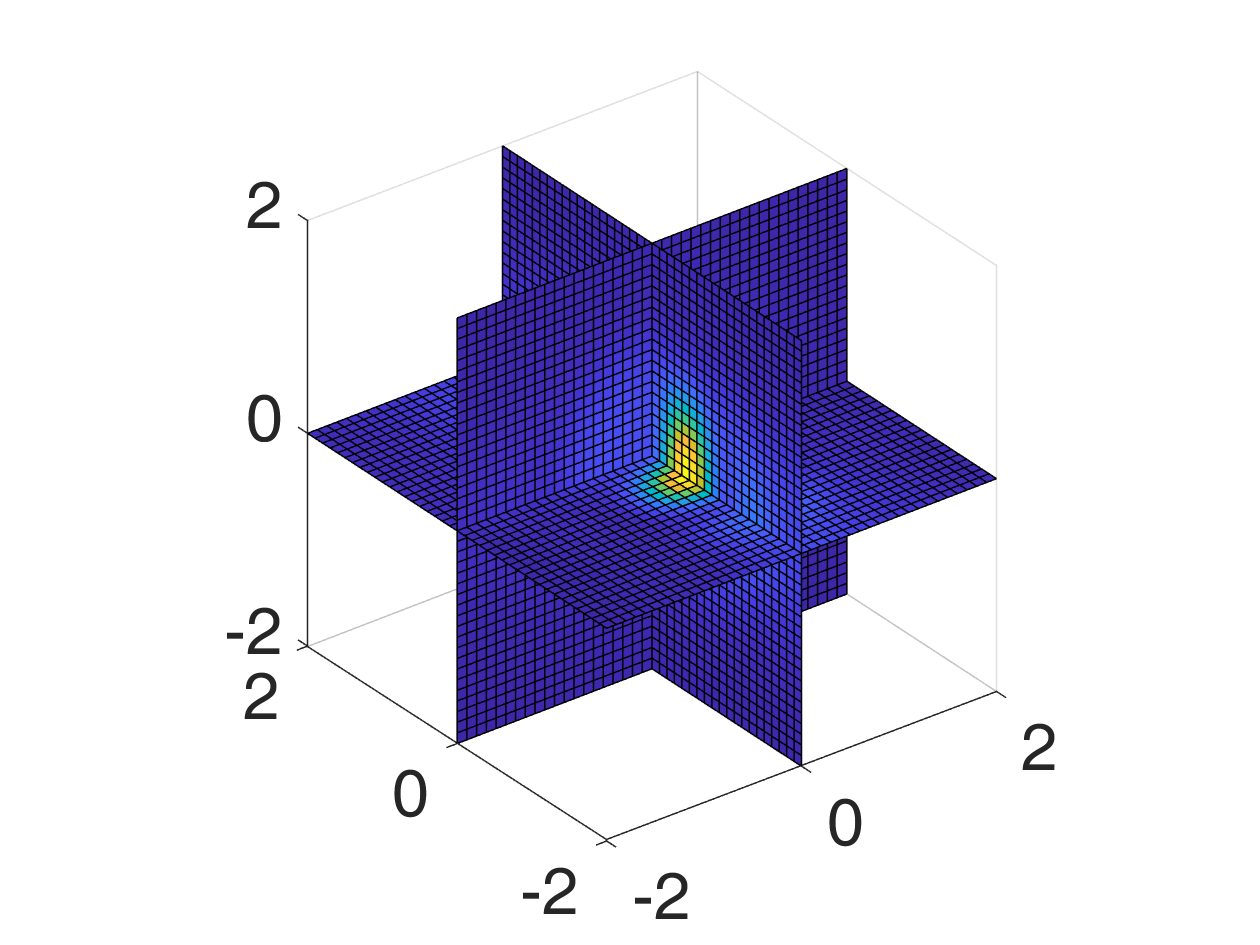}
	}
	\quad
	\subfigure[isosurface plot of $\mathcal{I}(z)$]{
		\includegraphics[height=5cm]{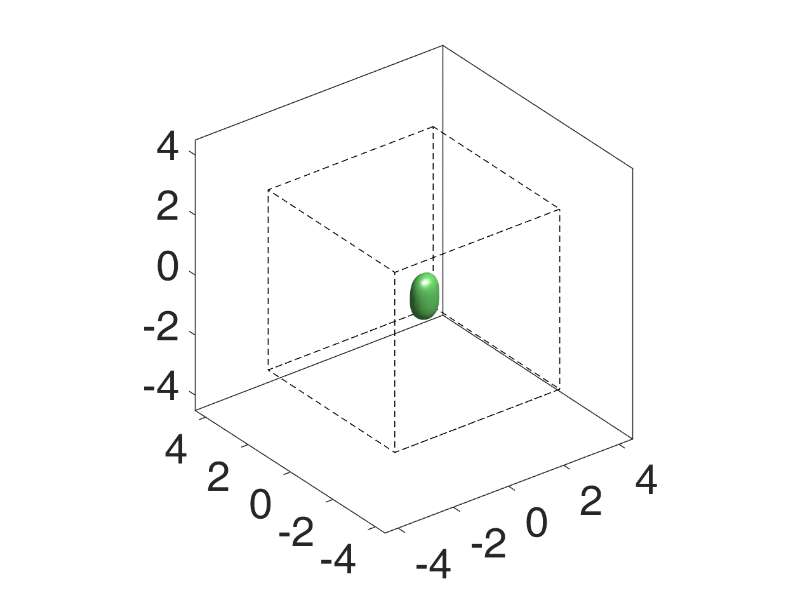}
	}
	\caption{\label{fig6} (a) The geometry setting of a cube scatterer with a side length $d=1$;  (b) plot of the receivers on one surface;  (c)-(d) slices and isosurface plot of the imaging functional $\mathcal{I}(z)$, respectively.}
\end{figure}

\begin{figure}[h]
	\centering
	
	\subfigure[geometry setting]{
		\includegraphics[height=5cm]{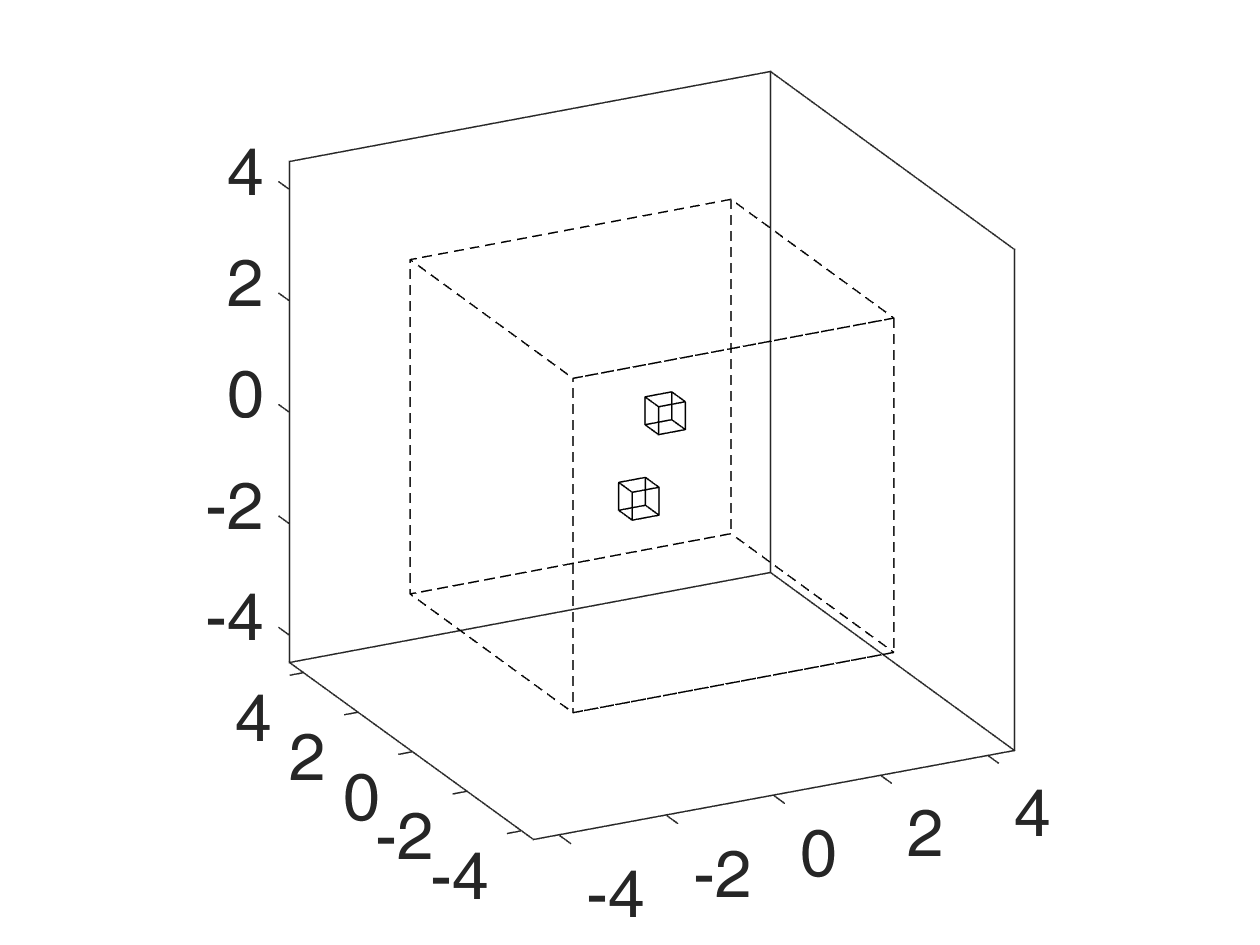}
	}
	\quad
	\subfigure[isosurface plot of $\mathcal{I}(z)$]{
		\includegraphics[height=5cm]{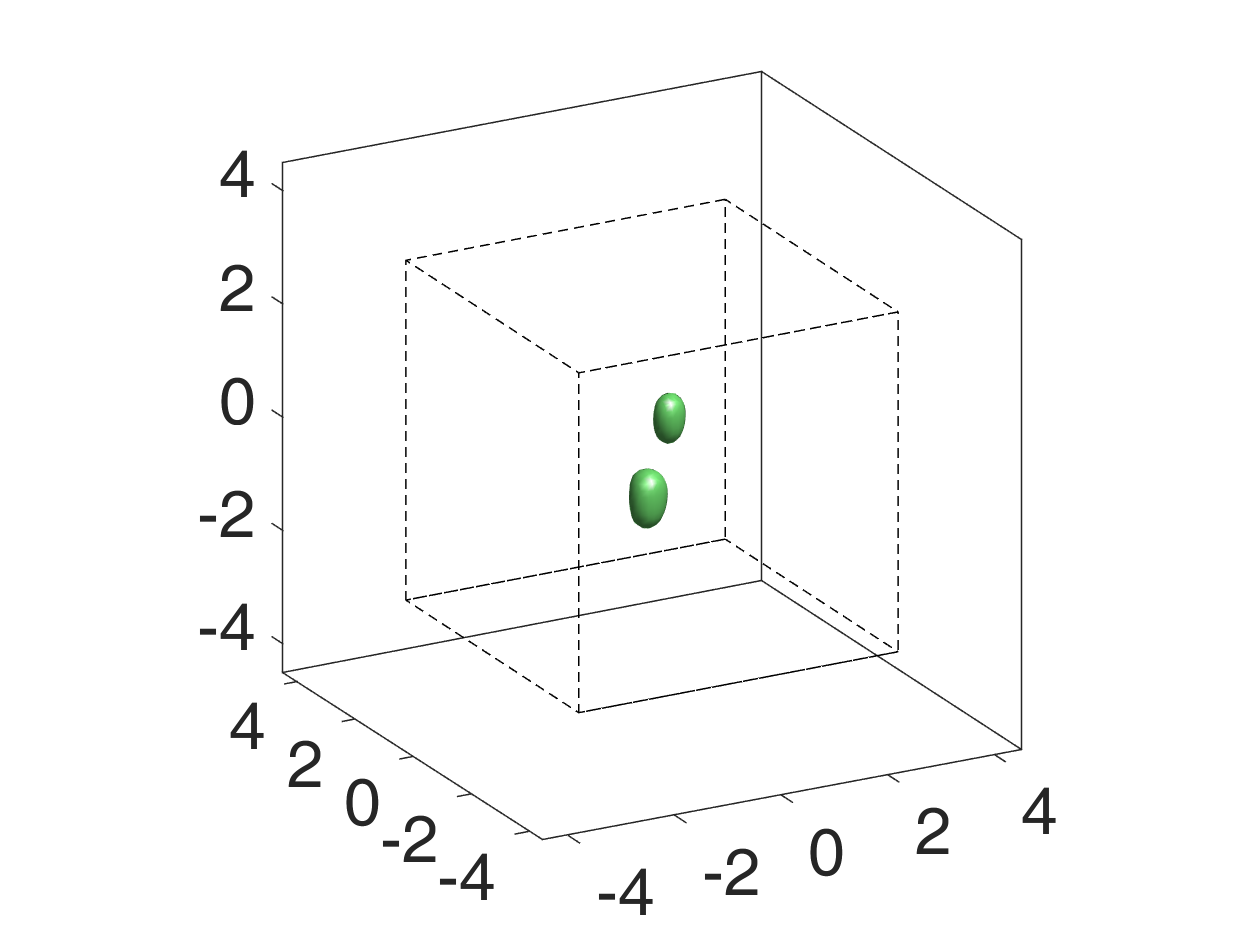}
	}
	\quad
	\subfigure[slices plot of $\mathcal{I}(z)$]{
		\includegraphics[height=5cm]{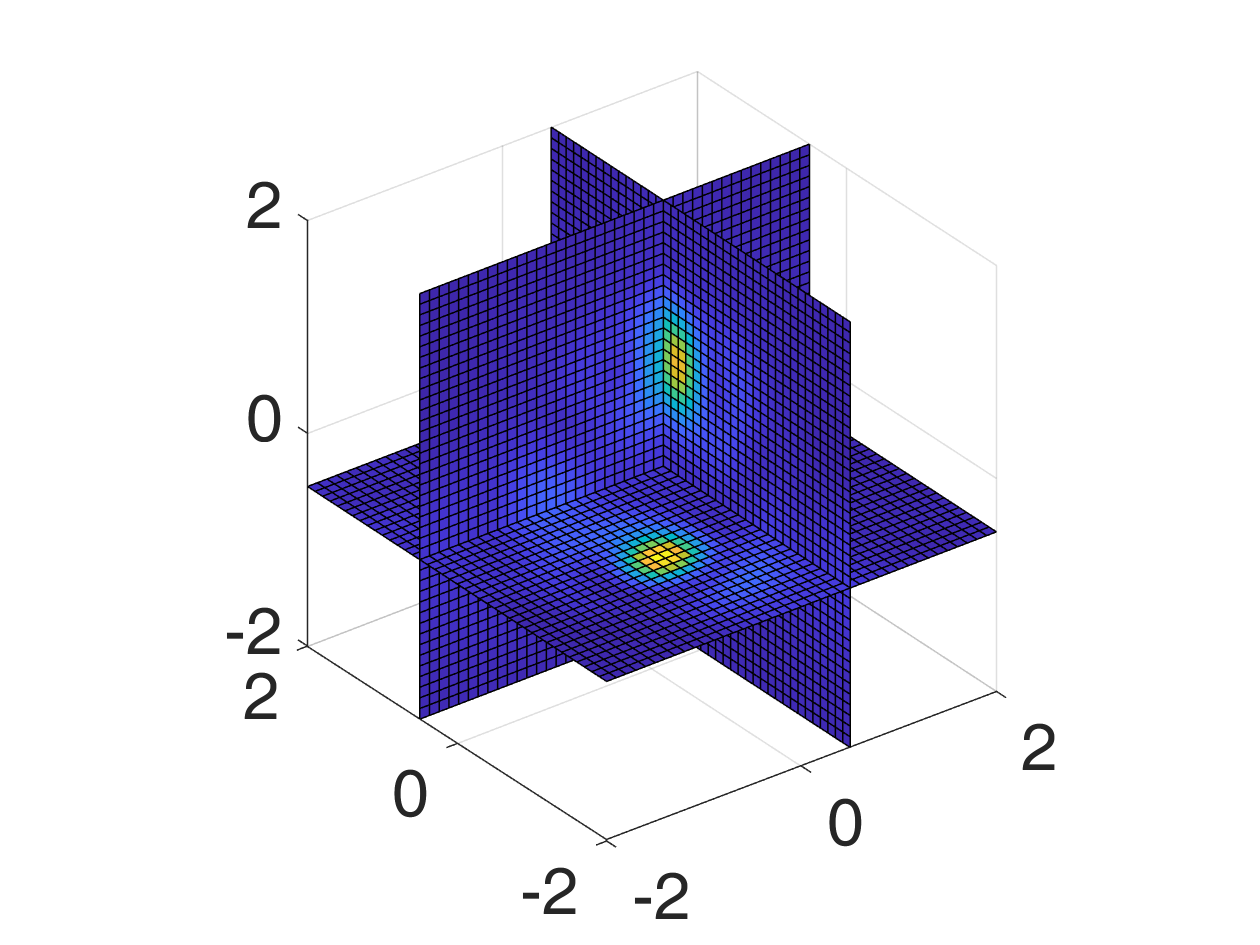}
	}
	\caption{\label{fig7} (a) The geometry setting of two cube scatterers with a side length $d=0.5$; (b)-(c) isosurface and slices plot of the imaging functional $\mathcal{I}(z)$, respectively.}
\end{figure}

\begin{rem}
It is noteworthy that our approach can efficiently determine the locations of the cubes, but it does not achieve satisfactory performance in shape reconstruction. This limitation arises from our use of only a single incident source, whereas the reconstruction results can be significantly improved by employing multiple incident sources.  To maintain consistency with the theoretical justification, we do not present the reconstructions using multiple incident sources. Interested readers can refer to the numerical results for acoustic waves with multiple sources in \cite{guo2023}.

\end{rem}


\end{document}